\documentclass[10pt,a4paper]{amsart}
\usepackage{graphicx, amssymb, xcolor,pdfsync,enumitem}
\usepackage[all,cmtip]{xy}
\numberwithin{equation}{section}
\numberwithin{figure}{section}
\usepackage{mathrsfs}
\usepackage{tikz}
\usepackage{tikz-cd}
\usepackage{faktor}
\usepackage{setspace}
\usepackage{hyperref}
\usepackage{graphicx}
\usepackage{yhmath}
\usepackage{mathdots}
\usepackage{MnSymbol}

\hypersetup{
    colorlinks=true,       
    linkcolor=blue,          
    citecolor=blue,        
    filecolor=blue,      
    urlcolor=blue           
}
\usepackage{tikz}
\usetikzlibrary{matrix,arrows}

\usepackage[switch]{lineno}
\usepackage{yfonts}

\advance\hoffset by -0.9 truecm

\begin{document}

\newcommand{\s}{\vspace{0.2cm}}
\newcommand{\Sp}{\mbox{Sp}}
\newcommand{\sy}{\mbox{sym}}
\newcommand{\GP}{\mathcal{G}_p}
\newcommand{\FP}{\mathcal{F}_p}
\newcommand{\G}{\mathbb{Z}_p^2 \rtimes \mathbf{D}_3}
\newcommand{\Z}{\mathbb{Z}}
\newcommand{\D}{\mathbf{D}}
\newcommand{\Q}{\mathcal{Q}}
\newcommand{\HH}{\mathbb{H}}
\newcommand{\oph}{orientation-preserving homeomorphism}

\newtheorem{theo}{Theorem}
\newtheorem{prop}{Proposition}
\newtheorem{coro}{Corollary}
\newtheorem{lemm}{Lemma}
\newtheorem{claim}{Claim}
\newtheorem{conj}{Conjecture}
\newtheorem{example}{Example}
\theoremstyle{remark}
\newtheorem{rema}{\bf Remark}

\newtheorem*{rema1}{\bf Remark}
\newtheorem*{defi}{\bf Definition}
\newtheorem*{theo*}{\bf Theorem}
\newtheorem*{coro*}{Corollary}
\newtheorem*{conj*}{Conjecture}

\newtheorem{hecho}[theo]{\bf Hecho}

\title{$\mathbb{Z}_k^m$-actions of signature $(0; k, \stackrel{n+1}{\ldots}, k)$}
\date{}

\author{Rub\'en A. Hidalgo}
\address{Departamento de Matem\'atica y Estad\'{\i}stica,  Universidad de La Frontera,  Francisco Salazar 01145, Temuco, Chile}
\email{ruben.hidalgo@ufrontera.cl}

\author{Sebasti\'an Reyes-Carocca}
\address{Departamento de Matem\'aticas, Facultad de Ciencias, Universidad de Chile, Las Palmeras 3425, Santiago, Chile}
\email{sebastianreyes.c@uchile.cl}

\thanks{Partially supported by ANID Fondecyt Regular  Grants 1230001,  1220099 and 1230708}

\begin{abstract} 
An action of a finite group $G$ is a pair $(S,\hat{G})$, where $S$ is a compact Riemann surface of genus $g \geqslant 2$ and $\hat{G} \leqslant {\rm Aut}(S)$ is isomorphic to $G$.
To each action $(S,\hat{G})$ there is associated a signature $(\gamma;k_{1},\ldots,k_{r})$ that codifies the orbifold structure of $S/\hat{G}$.
Two actions of $G$, say $(S_{1},G_{1})$ and $(S_{2},G_{2})$, are topologically equivalent if there is an orientation-preserving homeomorphism $\varphi:S_{1} \to S_{2}$ such that $\varphi G_{1} \varphi^{-1}=G_{2}$.  Topologically equivalent actions necessarily must have the same signature.  The problem of determining the number of different topological actions of $G$ for a given signature is in general a difficult task.
In this article, we describe, up to topological equivalence, those actions when $G$ is an abelian group and quotient genus $\gamma=0$. We are particularly interested in the case $G={\mathbb Z}_{k}^{m}$ and the quotient signature of the action to be of the form $(0;k,\stackrel{n+1}{\ldots},k)$.
\end{abstract}
\maketitle
\thispagestyle{empty}

\section{Introduction} 
Riemann surfaces have been proved to serve as fundamental objects bridging topology, geometry, algebra and complex analysis. The study of automorphisms of compact Riemann surfaces or,  equivalently, of smooth complex projective algebraic curves and their function fields, represents a classical and rich area of research in both complex and algebraic geometry. The foundations of this field {date back} to the nineteenth century, with seminal contributions from mathematicians such as Riemann, Klein and Jacobi. A central result, due to Schwarz \cite{Schwarz} and Hurwitz \cite{Hurwitz}, establishes that the group ${\rm Aut}(S)$ of conformal automorphisms of a compact Riemann surface $S$ of genus $g \geqslant 2$ is finite, and that its order is bounded by $84(g-1)$. In \cite{Greenberg}, Greenberg succeeded in proving that every finite group $G$ can be realised as a group of automorphisms of some compact Riemann surface of genus $g \geqslant 2$.

A {\it $G$-action of genus $g$} is a pair $(S,\hat{G})$, where $S$ is a compact Riemann surface of genus $g$ and $\hat{G} \leqslant {\rm Aut}(S)$ is isomorphic, as abstract group, to $G$. 
The {\it signature} of the action is the tuple $(\gamma; k_1, \ldots, k_r)$, where $\gamma$ is the genus of the quotient $S/\hat{G}$ and $2 \leqslant k_{1} \leqslant  \ldots \leqslant  k_r$ are the branch indices of the canonical projection $S \to S/\hat{G}$ (the integers $k_j$ are the orders of the cone points of the Riemann orbifold $S/\hat{G}$). We also say that the above is a {\it $G$-action of signature} $(\gamma;k_{1},\ldots,k_{r})$. Two $G$-actions $(S_{1},G_{1})$ and $(S_{2},G_{2})$ are called {\it topologically equivalent} if there exists an 
orientation-preserving homeomorphism 
$$\varphi : S_{1} \to  S_{2} \; \mbox{ such that } \varphi^{-1} G_{2}\varphi=G_{1}.$$ Note that topologically equivalent $G$-actions must be of the same genus and must have the same signature, but the converse is not in general true (for instance, there are two topologically inequivalent ${\mathbb Z}_{7}$-actions of genus three of signature $(0;7,7,7)$).

The importance of the topological classification of actions lies in several applications, some of which we briefly describe.
The moduli space $\mathscr{M}_g$, of isomorphism classes of compact Riemann surfaces of genus $g \geqslant 2$, has the structure of a complex analytic space of dimension $3g-3$. If 
$g \geqslant 4$, then its singular locus corresponds to the points representing compact Riemann surfaces with non-trivial automorphisms. i.e., 
$$\mbox{Sing}(\mathscr{M}_g)=\{[S] \in \mathscr{M}_g : \mbox{Aut}(S) \neq \{\mbox{id}\}\}.$$

 The moduli space --which is itself an algebraic variety defined over the field of  rational numbers-- is one of the most fascinating objects in algebraic geometry, and is at the core of important and recent developments in number theory and geometry. Let us fix a $G$-action $(S_{0},G_{0})$ of genus $g$.
The set $$\mathscr{M}_g(S_{0},G_{0})=\{[S] \in \mathscr{M}_g: \ \text{there is a $G$-action } (S, \hat{G}) \ \text{topological equivalent to } (S_{0},G_{0})\} \subset \mbox{Sing}(\mathscr{M}_g),$$ 
is an irreducible subvariety of the moduli space \cite{GG92}. In other words, the topological equivalence allows us to perform deformations of Riemann surfaces with automorphisms in a controlled manner. These subvarieties, in turn, provide a stratification of the moduli space \cite{B90}, which has been proved to serve as a useful toolkit to explore the largely unknown topology of $\mathscr{M}_g.$ For instance, it has played a key role in the determination of the connectedness of the singular locus of the moduli space (see, for instance, \cite{BCI} and the references therein) and has been recently employed in \cite{nos} to find new non-normal subvarieties of $\mathscr{M}_g$. Moreover, these subvarieties have been fruitful to study some aspects of families of Jacobian varieties with group action; for instance, isogeny decompositions and Shimura varieties (see, for example, \cite{F15} and \cite{IRR}).  Another important aspect to mention is that the conjugacy classes of finite subgroups of the mapping class group (homotopy classes of self-homeomorphisms) of a real orientable compact surface of genus $g$ are in bijective correspondence with the topological classes of actions of genus $g$. 
Sources for the characterization of topological actions by purely algebraic data include Nielsen \cite{Nielsen}, Harvey \cite{Harvey} and Gilman \cite{Gilman}. We also refer to \cite{BRT}, \cite{CGP}, \cite{K}, \cite{P23}, \cite{RR22} and \cite{LMFDB} as sources for the classification of topological actions in low genera, and also for computer-aided algorithms.

Determining the number of distinct topological $G$-actions for fixed genus $g \geqslant 2$ is, in general, a challenging problem. A first step is to compute all possible signatures for these $G$-actions. Secondly, for each such  signature $(\gamma;k_{1},\ldots,k_{r})$, one needs to compute the different topological $G$-actions with that signature. To compute all possible topologically different $G$-actions with the given signature, one may proceed as follows (see, for instance, \cite{BW07}). Consider the group 
$$\Gamma=\langle a_{1},b_{1},\ldots,a_{\gamma},b_{\gamma},x_{1},\ldots,x_{r}: x_{1}^{k_{1}}=\cdots=x_{r}^{k_{r}}=\prod_{j=1}^{\gamma}[a_{j},b_{j}] \prod_{i=1}^{r}x_{i}=1\rangle.$$
Each topological $G$-action, with the above given signature, corresponds to a torsion-free finite index normal subgroup $F$ of $\Gamma$ such that $\Gamma/F \cong G$. We denote by ${\mathcal F}$ the (finite) collection formed by such subgroups. If $F_{1}, F_{2} \in {\mathcal F}$ then they produce the same topological action if and only if there is a geometric automorphism $\psi$ of $\Gamma$ such that $F_{2}=\psi(F_{1})$.  These geometric automorphisms form an index two subgroup $B_{\Gamma}$ of ${\rm Aut}(\Gamma)$. As the group $B_{\Gamma}$ is infinite, this poses a challenge for practical computations. A natural question that arises is whether the infinite group 
$B_{\Gamma}$ can be replaced by a finite group, thereby rendering the computations feasible. If $G$ is abelian and $\gamma=0$, we will see how to deal with this problem (see Corollary \ref{abeliano3} and also \cite{BW07}).

In this article, we are particularly concerned with the above task for ${\mathbb Z}_{k}^{m}$-actions of signature $(0;k,\stackrel{n+1}{\ldots},k)$, where $n, k \geqslant 2$ and $1 \leqslant m \leqslant n$ are integers such that $(n-1)(k-1)>2$. As we shall see later, roughly speaking, in this situation we can replace the infinite group $\Gamma$ by ${\mathbb Z}_{k}^{n}$, and $B_{\Gamma}$ by a subgroup of ${\rm Aut}({\mathbb Z}_{k}^{n})$ which is isomorphic to the symmetric group $\mathbf{S}_{n+1}$. The first result of this paper is given in 
Theorem \ref{teorema2} and Corollary \ref{felix}, where we provide the description of the different topological ${\mathbb Z}_{k}^{m}$-actions of signature $(0;k,\stackrel{n+1}{\ldots},k).$ After that, we consider triples $(S,N,G)$, where $(S,N)$ is a ${\mathbb Z}_{k}^{m}$-action of signature $(0;k,\stackrel{n+1}{\ldots},k)$ and $N$ is a normal subgroup of $G \leqslant {\rm Aut}(S)$. Note that the quotient group $G/N$ induces a permutation of the $n+1$ cone points of $S/N$. We consider all those triples yielding the same permutational action, and say that two of them, say $(S_1,N_1,G_1)$ and $(S_2,N_2,G_2)$, are {\it topologically equivalent} if there exists an 
orientation-preserving homeomorphism $\varphi : S_{1} \to  S_{2}$ such that $\varphi^{-1} N_{2}\varphi=N_{1}$ and $\varphi^{-1} G_{2}\varphi=G_{1}.$ The second result of this paper is given in Theorem \ref{teorema3} and Corollary \ref{coca7}, where we describe these different topological triples. It is worth noting that two such triples might be topologically non-equivalent, and either the pairs $(S_{1},N_{1})$ and $(S_{2},N_{2})$ or the pairs $(S_{1},G_{1})$ and $(S_{2},G_{2})$ to be topologically equivalent (even, these pairs might be topologocally equivalent by using different homeomorphisms, but not for a common one).

For the sake of clarity, we work out the case of $\mathbb{Z}_p^2$-actions of signature $(0; p, \stackrel{n+1}{\ldots}, p)$, where $p$ is a prime integer. We provide explicit algebraic models for these Riemann surfaces in terms of fiber products of cyclic $p$-gonal algebraic curves, and give an isogeny decomposition of their Jacobian varieties. After that, we specialise our results for the cases $n =3$ and $n=5,$ as they are interesting in their own right. More precisely, we study $\mathbb{Z}_p^2$-actions of signature $(0; p, \stackrel{4}{\ldots}, p)$ and describe some examples in detail, and study $\mathbb{Z}_p^2$-actions of signature $(0; p, \stackrel{6}{\ldots}, p)$ that admit extra automorphisms and that form complex one-dimensional families. By the way, we recover and extend classical and recent results for some families of Riemann surfaces.  

\s

This paper is organised as follows. In \S\ref{prelis} we introduce some notation and briefly review some preliminaries concerning group actions on Riemann surfaces. In \S\ref{Sec:Abeliano}, we restrict to the case of abelian actions. After that, in \S\ref{CFG}, we review known results concerning generalised Fermat curves that will be used throughout the paper. In \S\ref{geoaut}, we introduce geometric automorphisms and relate them with a group of permutations. The main results are stated and proved in \S\ref{state}.

\section{Preliminaries}\label{prelis}

\subsection{Riemann orbifolds}
A compact {\it Riemann orbifold} ${\mathcal O}$ consists of a compact Riemann surface $S$, called its underlying Riemann surface structure, together a (possible empty) finite collection $\{(q_{j},k_{j})$\}, where $q_{j} \in S$, called its {\it cone points}, and $k_{j} \geqslant 2$ an integer, called the {\it branch order} of $q_{j}$. Up to permutation of indices, we may assume $2 \leqslant k_{1} \leqslant \cdots \leqslant k_{r}$. 
We say that the Riemann orbifold ${\mathcal O}$ has {\it signature} $(\gamma; k_{1},\ldots, k_{r})$.  If 
$\sum_{j=1}^{r} k_{j}^{-1} <2\gamma+r-2,$
then we say that the orbifold is {\it hyperbolic}.
Two Riemann orbifolds are called {\it topologically equivalent} (respectively, {\it biholomorphically equivalent}) if there is an orientation-preserving homeomorphism (respectively, a biholomorphism) between the underlying Riemann surfaces that sends the cone points of one to the cone points of the other, preserving the branch orders. 
We use the symbol ${\mathcal O}_{1} \cong {\mathcal O}_{2}$ to indicate that two Riemann orbifolds ${\mathcal O}_{1}$ and ${\mathcal O}_{2}$ are biholomorphically equivalent. If $k \geqslant 2$ is an integer, the symbol $(0; k^n)$ abbreviates the signature $(0; k, \stackrel{n}{\ldots},k).$

\subsection{Geometric automorphisms}
Let $\mathbb{H}$ denote the upper half-plane, and 
let $\Gamma_{1}$ and $\Gamma_{2}$ be Fuchsian groups of the same signature $(\gamma;k_{1},\ldots,k_{r})$. 
By Nielsen's theorem, if $\eta:\Gamma_{1} \to \Gamma_{2}$ is a group isomorphism, then there is a homeomorphism $\varphi:{\mathbb H} \to {\mathbb H}$  such that $\eta(a)=\varphi \circ a \circ \varphi^{-1}$, for $a \in \Gamma_{1}$. If the homeomorphism $\varphi$ is orientation-preserving, then we will say that the isomorphism $\eta$ is a {\it geometric isomorphism}.
In particular, if $\Gamma$ is a Fuchsian group as above, then its group $B_{\Gamma}$ of geometric automorphisms is an index two subgroup of its group of automorphisms ${\rm Aut}(\Gamma)$.

\subsection{Group actions on Riemann surfaces}
Let $G$ be a finite group and $(S,\hat{G})$ be a $G$-action of genus $g \geqslant 2$.
This $G$-action corresponds to an injective homomorphism $\theta:G \to {\rm Aut}(S)$, where $\theta(G)=\hat{G}$ (any other choice is of the form $\theta \circ \psi$, where $\psi$ is an automorphism of $G$).
Associated to the above $G$-action is the quotient Riemann orbifold $S/\hat{G}$, whose underlying Riemann surface structure $R$ has genus $\gamma \geqslant 0$. A point $q_{j} \in R$ is a cone point of branch order $k_{j} \geqslant 2$ if it is the $\hat{G}$-equivalence class of a point in $S$ with $\hat{G}$-stabilizer of order $k_{j} \geqslant 2$. So, its signature is $(\gamma; k_{1},\ldots, k_{r})$, which is the signature of the $G$-action.
If $\pi:S \to S/\hat{G}$ is a regular covering map with deck group $\hat{G}$, then the cone points correspond to the branch values of $\pi$ and $k_{j}$ is the local degree of $\pi$ at a point $x \in \pi^{-1}(q_{j})$. 
By the Riemann-Hurwitz formula one has that
$2(g-1)=|\hat{G}| \left(2\gamma + r-2-\sum_{j=1}^{r} k_{j}^{-1} \right).$
As we are assuming $g \geqslant 2$, the above asserts that the orbifold $S/\hat{G}$ is hyperbolic.
By the classical uniformization theorem, there exists a co-compact Fuchsian group $\Gamma \leqslant {\rm PSL}_{2}({\mathbb R}) \cong \mbox{Aut}({\mathbb H})$ such that  $S/\hat{G}$ and ${\mathbb H}/\Gamma$ are biholomorphic as Riemann orbifolds ($\Gamma$ is unique up to conjugation by elements of ${\rm PSL}_{2}({\mathbb R})$). 
This group has a presentation 
\begin{equation}\label{grupoGamma}
\Gamma=\langle a_{1},b_{1},\ldots,a_{\gamma},b_{\gamma},x_{1},\ldots,x_{r}: x_{1}^{k_{1}}=\cdots=x_{r}^{k_{r}}=\prod_{j=1}^{\gamma}[a_{j},b_{j}] \prod_{i=1}^{r}x_{i}=1\rangle,
\end{equation}
so, $\Gamma$ has signature $(\gamma;k_{1},\ldots,k_{r})$.
Moreover, there is a  group epimorphism  $\vartheta:\Gamma \to G$, whose kernel $F$ is torsion-free, in such a way that there is a biholomorphism $f:S \to {\mathbb H}/F$ such that $f \hat{G} f^{-1} = \Gamma/ F$. The epimorphism $\vartheta$ is called  {\it a monodromy} of the action (which is unique up to post-composing with automorphisms of $G$).  In this way, we may identify the collection of $G$-actions, of the above signature, with the collection of torsion-free finite index normal subgroups $F$ of $\Gamma$ such that $\Gamma/F \cong G$.

\subsection {Topological equivalence of actions}
We recall that two actions $(S_{1},G_{1})$ and $(S_{2},G_{2})$ of the finite group $G$ are {\it topologically equivalent} (respectively, {\it biholomorphically equivalent}) if there exists an orientation-preserving homeomorphism (respectively, biholomorphism)
$\varphi : S_{1} \to  S_{2}$ such that $\varphi^{-1} G_{2}\varphi=G_{1}$. In terms of monodromies, the topological equivalence can be reformulated as follows. 
 Let $\Gamma_{j}$ be a Fuchsian group such that ${\mathbb H}^{2}/\Gamma_{j} \cong S_{j}/G_{j}$, and let 
$\vartheta_{j}:\Gamma_{j} \to G$ be a corresponding monodromy of the action $(S_{j},G_{j})$. In this setting, the topologically equivalence of these two pairs is equivalent to  the existence of an automorphism $\rho$ of $G$ and a geometric isomorphism $\psi: \Gamma_{1} \to \Gamma_{2}$ such that $\rho \circ \vartheta_{1}=\vartheta_{2} \circ \psi.$ In terms of the associated torsion-free finite index subgroups $F_{j}=\ker(\vartheta_{j})$, this corresponds to have $\psi(F_{1})=F_{2}$. Note that,  up to topological equivalence, we may assume $\Gamma_1=\Gamma_2=\Gamma$. In this case, all the above assert the following description (we refer to the survey \cite{Bs} for more details).

\begin{theo}[see, for instance, \cite{BW07}, \cite{Bs}]
Let $\Gamma$ be a Fuchsian group of signature $(\gamma;k_{1},\ldots,k_{r})$ and let $G$ be a finite group. 
Then the number of topologically different $G$-actions of signature $(\gamma;k_{1},\ldots,k_{r})$ is equal to the number of surjective homomorphisms $\vartheta:\Gamma \to G$ with torsion-free kernel,  up to post-composition by elements of ${\rm Aut}(G)$ and pre-composition by elements of $B_{\Gamma}$. Equivalently, if ${\mathcal F}$ is the collection of those torsion-free finite index normal subgroups $F$ of $\Gamma$ such that $\Gamma/F \cong G$, then the number of topologically different actions of $G$ with signature $(\gamma;k_{1},\ldots,k_{r})$ is equal to the cardinality of ${\mathcal F}/{\rm B}_{\Gamma}$.
\end{theo}

 In the case that $\gamma=0$, a set of generators for the group of geometric automorphisms $B_{\Gamma}$ of $\Gamma$ is known.



\section{Abelian Actions}\label{Sec:Abeliano}
In this section, we restrict to those actions of a given finite abelian group $A$. 

\subsection{}
Let $(S,N)$ be a $A$-action of (hyperbolic) signature $(\gamma;k_{1},\ldots,k_{r})$. Let $\Gamma$ be a Fuchsian group such that the Riemann orbifolds $S/N$ and ${\mathbb H}/\Gamma$ are biholomorphic; so $\Gamma$ has a presentation as in \eqref{grupoGamma}. There is a torsion-free finite index normal subgroup $F$ of $\Gamma$ such that there is biholomorphism $\varphi:S \to {\mathbb H}/F$ satisfying that $\varphi N \varphi^{-1}=\Gamma/F$. Since $A \cong \Gamma/F$ is abelian, the commutator subgroup $\Gamma'$ of $\Gamma$ must be contained inside $F$. This asserts that $\Gamma'$ is torsion-free. In \cite{Maclachlan}, Maclachlan proved that $\Gamma'$ is torsion free if and only if 
\begin{equation}\label{eq2}
{\rm lcm}(k_{1},\ldots,k_{j-1},k_{j+1},\ldots,k_{r})=m_{0}, \quad j=1,\ldots,r,
\end{equation}
where ${\rm lcm}$ denotes the least common multiple. As a consequence of the above, one has the following result.

\begin{theo}\label{teoremaiso1}
Let $A$ be a finite abelian group.
If $(S,N)$ is an $A$-action of (hyperbolic) signature $(\gamma;k_{1},\ldots,k_{r})$, then
(1)  the collection $\{k_{1},\ldots,k_{r}\}$ satisfies the Maclachlan condition \eqref{eq2}, and 
(2) there is a Fuchsian group $\Gamma$ with presentation \eqref{grupoGamma} and a torsion-free finite index normal subgroup $F$, containing $\Gamma'$, such that $(S,N)$ and $({\mathbb H}/F,\Gamma/F)$ are biholomorphically equivalent.
\end{theo}

\subsection{}
Let $(S_{1},A_{1})$ and $(S_{2},A_{2})$ be topologically equivalent $A$-actions, say with signature $(\gamma;k_{1},\ldots,k_{r})$.
Let $\Gamma_{j}$ be a Fuchsian group with presentation \eqref{grupoGamma} and $F_{j}$ be a torsion-free finite index normal subgroup of $\Gamma_{j}$, containing $\Gamma_{j}'$, such that $(S_{j},A_{j})$ and $({\mathbb H}/F_{j},\Gamma_{j}/F_{j})$ are biholomorphically equivalent.
So, there is an orientation-preserving homeomorphism $\phi:{\mathbb H}/F_{1} \to {\mathbb H}/F_{2}$ such that $\phi (\Gamma_{1}/F_{1}) \phi^{-1}=\Gamma_{2}/F_{2}$. This homeomorphism descends to an orientation-preserving homeomorphism $\hat{\phi}:{\mathbb H}/\Gamma_{1} \to {\mathbb H}/\Gamma_{2}$.
Note that $\hat{\phi}$ sends the set of cone points of ${\mathbb H}/\Gamma_{1}$ onto the set of cone points of ${\mathbb H}/\Gamma_{2}$ preserving the branch orders. 
Now, as $\Gamma_{j}'$ is a characteristic subgroup of $\Gamma_{j}$, the homeomorphism $\hat{\phi}$ lifts to an orientation-preserving homeomorphism $\tilde{\phi}:{\mathbb H}/\Gamma_{1}' \to {\mathbb H}/\Gamma_{2}'$ such that $\tilde{\phi} (F_{1}/\Gamma_{1}') \tilde{\phi}^{-1}=F_{2}/\Gamma_{2}'$ and $\tilde{\phi} (\Gamma_{1}/\Gamma_{1}') \tilde{\phi}^{-1} = \Gamma_{2}/\Gamma_{2}'$.
All the above permits us to observe the following.

\begin{theo}\label{abeliano1}
Let $A$ be a finite group and let $\Gamma$ be a fixed Fuchsian group of signature $(\gamma;k_{1},\ldots,k_{r})$, where the integers $k_{1},\ldots,k_{r}$ satisfy the Maclachlan condition \eqref{eq2}.
If $(S,N)$ an $A$-action of signature $(\gamma;k_{1},\ldots,k_{r})$, then there exists a torsion-free finite index normal subgroup $F$ of $\Gamma$, containing $\Gamma'$, such that 
$(S,N)$ and $({\mathbb H}/F,\Gamma/F)$ are topologically equivalent.
In particular, there is a finite index normal subgroup $K=F/\Gamma'$ of $H_{\Gamma}=\Gamma/\Gamma'$, acting freely on the Riemann surface $C_{\Gamma}={\mathbb H}/\Gamma'$, such that
$(S,N)$ and $(C_{\Gamma}/K,H/K)$ are topologically equivalent.
\end{theo}

\begin{rema}
If $\gamma=0$, then $H_{\Gamma}$ is a finite abelian group and the Riemann surface $C_{\Gamma}$ is compact.
\end{rema}

\subsection{}
Let us consider a fixed Fuchsian group $\Gamma$ of signature $(\gamma;k_{1},\ldots,k_{r})$, where the integers $k_{1},\ldots,k_{r}$ satisfy the Maclachlan condition \eqref{eq2}. Let us consider the Riemann surface $C_{\Gamma}={\mathbb H}/\Gamma'$ (called the homology cover of the orbifold ${\mathbb H}/\Gamma$), and the abelian group $H_{\Gamma}=\Gamma/\Gamma'$ (which is isomorphic to the orbifold homology group of the orbifold). 
If $A$ is a finite abelian group such that there is an $A$-action of the above signature, then we may consider the finite set
$${\mathcal F}_{A}(\gamma;k_{1},\ldots,k_{r})=\left\{K \leqslant H_{\Gamma}: \mbox{$K$ acts freely on $C_{\Gamma}$ and $H_{\Gamma}/K \cong A$}     \right\}.$$

As a consequence of Theorem \ref{abeliano1}, we observe the following.

\begin{coro}\label{corolario1}
Let $A$ be a finite abelian group.
Up to topological equivalence, every $A$-action of signature $(\gamma;k_{1},\ldots,k_{r})$ is given as $(C_{\Gamma}/K,H_{\Gamma}/K)$, where $K \in {\mathcal F}_{A}(\gamma;k_{1},\ldots,k_{r})$. 
\end{coro}

In particular, the above asserts that the cardinality of ${\mathcal F}_{A}(\gamma;k_{1},\ldots,k_{r})$ provides an upper bound for the number of actions of the group $A$ with the signature $(\gamma;k_{1},\ldots,k_{r})$.
Now, we determine when two members of ${\mathcal F}_{A}(\gamma;k_{1},\ldots,k_{r})$ determine topologically equivalent actions.

\begin{theo}\label{abeliano2}
Let $A$ a finite abelian group and $K_{1}, K_{2} \in {\mathcal F}_{A}(\gamma;k_{1},\ldots,k_{r})$. Then the actions
$(C_{\Gamma}/K_{1},H_{\Gamma}/K_{1})$ and $(C_{\Gamma}/K_{2},H_{\Gamma}/K_{2})$ are topologically equivalent if and only if there is an orientation-preserving homeomorphism $f:C_{\Gamma} \to C_{\Gamma}$ such that $f K_{1} f^{-1}=K_{2}$ and $f H_{\Gamma} f^{-1}=H_{\Gamma}$.
\end{theo}
\begin{proof}
(1) Assume the existence of $f$ as in the hypothesis. Then $f$ descends to an orientation-preserving homeomorphism $\tilde{f}:C_{\Gamma}/K_{1} \to C_{\Gamma}/K_{2}$ such that $\tilde{f} H_{\Gamma}/K_{1} \tilde{f}^{-1}=H_{\Gamma}/K_{2}$.
(2) In the other direction, assume there is an orientation-preserving homeomorphism 
$\tilde{f}:C_{\Gamma}/K_{1} \to C_{\Gamma}/K_{2}$ such that $\tilde{f} H_{\Gamma}/K_{1} \tilde{f}^{-1}=H_{\Gamma}/K_{2}$. Then, $\tilde{f}$ descends to an orientation-preserving self-homeomorphism $\hat{f}$ of the Riemann orbifold $C_{\Gamma}/H_{\Gamma}$. As $\Gamma'$ is characteristic subgroup of $\Gamma$, this homeomorphism lifts to an orientation-preserving homeomorphism $f:C_{\Gamma} \to C_{\Gamma}$ satisfying that  
$f H_{\Gamma} f^{-1}=H_{\Gamma}$ (since its descends to $\hat{f}$) and $f K_{1} f^{-1}=K_{2}$ (as its descends to $\tilde{f}$).
\end{proof}

\subsubsection{}
Let us denote by ${\rm Aut}_{g}(H_{\Gamma})$ the subgroup of ${\rm Aut}(H_{\Gamma})$ consisting of its geometric automorphisms, i.e., those $\Phi \in {\rm Aut}(H_{\Gamma})$ for which there exists an orientation-preserving homeomorphism $f:C_{\Gamma} \to C_{\Gamma}$ such that $\Phi(h)=f \circ h \circ f^{-1}$, for every $h \in H_{\Gamma}$.
Note that ${\rm Aut}_{g}(H_{\Gamma})$ keeps invariant the finite set ${\mathcal F}_{A}(\gamma;k_{1},\ldots,k_{r})$.

\begin{coro}\label{abeliano3}
Let $A$ be a finite abelian group. Then the number of topologically different $A$-actions of signature $(\gamma;k_{1},\ldots,k_{r})$ is the cvardinality of
${\mathcal F}_{A}(\gamma;k_{1},\ldots,k_{r})/{\rm Aut}_{g}(H_{\Gamma})$.
\end{coro}

For the sake of clarity, we include below an example.
\begin{example}
Let $A$ be a nontrivial abelian group and the signature $(0;2,2,3,3)$. Note that $|A|$ is divisible by $6$. In this case, 
$\Gamma=\langle x_{1},x_{2},x_{3},x_{4}: x_{1}^{2}=x_{2}^{2}=x_{3}^{3}=x_{4}^{3}=x_{1}x_{2}x_{3}x_{4}=1\rangle,$ and therefore 
$$H_{\Gamma}=\langle a_{1},a_{2},a_{3},a_{4}: a_{1}^{2}=a_{2}^{2}=a_{3}^{3}=a_{4}^{3}=a_{1}a_{2}a_{3}a_{4}=1, [a_{i},a_{j}]=1\rangle=\langle a_{1},a_{3}:a_{1}^{2}=a_{3}^{3}=[a_{1},a_{3}=1\rangle \cong {\mathbb Z}_{6}.$$Thus, $A={\mathbb Z}_{6}.$ 
Since every element of $H_{\Gamma}$ acts with fixed points, one has that ${\mathcal F}_{{\mathbb Z}_{6}}(0; 2,2,3,3)$ consists of only the trivial group, showing that there is only one topological action of $A$ with the above given signature.
\end{example}

In the following sections, we will consider the case when $A={\mathbb Z}_{k}^{m}$ and the signature is $(0;k^{n+1})$.

%
%
%

\section{Generalized Fermat curves}\label{CFG}
Let $n,k \geqslant 2$ be integers. A compact Riemann surface $X$ is called a
{\it generalized Fermat curve} of type $(k,n)$ if there exists a group $\mathbb{Z}_k^n \cong H_0 \leqslant
\mbox{Aut}(X)$ such that $X/H_0$ has signature $(0;k^{n+1})$. The group $H_0$ is a {\it generalized
Fermat group} of type $(k,n)$ and the pair $(X,H_0)$ is a {\it generalized
Fermat pair} of type $(k,n)$. By the Riemann-Hurwitz formula, the genus of $X$ is
$$g=1+ \tfrac{k^{n-1}}{2}((n-1)(k-1)-2).$$

In particular, the non-hyperbolic generalized Fermat pairs are those satisfying $(n-1)(k-1) \leqslant2$, that is, of type
$(2,2)$, $(2,3)$ and $(3,2)$. 
Henceforth, we restrict our attention to the hyperbolic case, that is, when $(n-1)(k-1)>2$.

\subsection{\bf Fuchsian description}\label{fdes} 
Let $(X, H_0)$ be a hyperbolic generalized Fermat pair of type $(k,n)$. The Riemann orbifold $X/H_0$ is uniformized by a Fuchsian group $\Gamma
= \langle x_1, \ldots, x_{n+1} : x_1^k = \ldots = x_{n+1}^k = x_1
\cdots x_{n+1} =1 \rangle.$ 
Let $\Gamma'$ stands for the commutator subgroup of $\Gamma$. Note that $(\mathbb{H}/  \Gamma',\Gamma / \Gamma')$ is a generalized Fermat pair of type $(k,n)$.

\begin{theo*}[\cite{GHL}]\mbox{}
\begin{enumerate}
\item  The generalized Fermat pairs 
$(X,H_0)$ and $(\mathbb{H}/  \Gamma',\Gamma / \Gamma')$ are biholomorphically equivalent.
\item The regular covering map $\pi:X \to X/H_{0}$ induced by the action of $H_{0}$ is characteristic.
\item Generalized Fermat pairs of the same type are topologically equivalent.
\item Let $S$ be a  Riemann surface of genus at least two  endowed with an abelian group of automorphisms $N$ such that $S/N$ has signature 
 $(0;k^{n+1})$. There exist a generalized Fermat pair $(X,H_{0})$ of type $(k,n)$ and a subgroup 
$K_S \lhd H_{0}$, acting freely on $X$, such that $(S, N)$ and $(X/K_S, H_{0}/K_S)$ are biholomorphically equivalent.
\end{enumerate}
\end{theo*}


\subsubsection{\bf Algebraic description and uniqueness of the generalized Fermat groups}
 In what follows, for each integer $k \geqslant 2$ we set $\omega_{k}=\exp(2 \pi i/k)$. If $n=2$, we set $\Omega_{2}=\{\emptyset\}$ and, for $n \geqslant 3$, we set 
 $$\Omega_n:=\{(\lambda_{1}, \ldots, \lambda_{n-2}) \in \mathbb{C}^{n-2}: \lambda_{j} \neq 0, 1 \mbox{ and } \lambda_{i} \neq \lambda_{i} \mbox{ for } i \neq j\}.$$ 
 
Let $(X,H_{0})$ be a generalized Fermat pair of type $(k,n)$. Then there is a regular covering map $\pi:X \to  \bar{\mathbb{C}}$, with deck group $H_{0}$, such that its set of branch values are 
 \begin{equation}\label{ram}
\left\{ \begin{array}{ll}
\infty, 0, 1, & \mbox{if $n=2$.} \\
\infty, 0, 1, q_{4},\ldots, q_{n+1}, & \mbox{if $n\geqslant 3$.}
\end{array}
\right.
\end{equation}
Note that, if $n \geqslant 3$, then $\Lambda:=(q_{4},\ldots,q_{n+1}) \in \Omega_{n}$. If $n=2$, we set $\Lambda=\emptyset$.



Consider the complex projective algebraic curve in
$\mathbb{P}^n$ defined by
$C_{k}(\emptyset)=\{x_{1}^{k} + x_{2}^{k} + x_{3}^{k} = 0 \}$ for $n=2$, and  
\begin{equation*}\label{cc1}C_k(\Lambda) :  \left \{ \begin{array}{ccccccc}
x_1^k & + & x_2^k & + & x_3^k & =  & 0\\
q_4 x_1^k & + & x_2^k & + & x_4^k & =  & 0\\
&\vdots& &\vdots & & \vdots\\
q_{n+1} x_1^k& + & x_2^k & + & x_{n+1}^k & =  & 0\\\end{array} \right\}\end{equation*}
for $n \geqslant 3$ (the fact that  $q_j \in \mathbb{C}-\{0,1\}$ are pairwise distinct implies that the above algebraic curve is non-singular, and hence represents a compact Riemann
surface). 

Note that the linear transformations $a_{j} \in {\rm PGL}_{n+1}({\mathbb C})$ defined as\begin{equation}\label{lin}a_j([x_1: \ldots :x_{n+1}])  = 
[x_1:\ldots:x_{j-1}:\omega_{k} x_j:x_{j+1}: \ldots: x_{n+1}] \mbox{ for }j \in \{ 1,\ldots,n \},
\end{equation}generate a group of automorphisms $H \cong {\mathbb Z}_{k}^{n}$ of $C_k(\Lambda).$ The
$k^n$-fold covering map \begin{equation}\label{pi} \pi :C_k(\Lambda) \to \bar{\mathbb{C}} \, \mbox{ given by }\, \pi ([x_1:\ldots:x_{n+1}])= - (x_2/x_1)^k,\end{equation} satisfies $
\pi \circ a_j = \pi$ for every $1\leqslant j \leqslant n$ and ramifies precisely over the points in \eqref{ram}. Thus,
$(C_k(\Lambda), H)$ is a generalized Fermat pair of type
$(k,n).$  

\begin{theo*}[\cite{GHL}]
The pairs $(X,H_0)$ and $(C_k(\Lambda) , H)$ are biholomorphically equivalent.
\end{theo*}

\begin{rema}
The above fact asserts that, for $n \geqslant 3$, the domain $\Omega_n$
parametrizes all generalized Fermat pairs of type $(k,n)$. 
\end{rema}

\begin{theo*}[\cite{GHL,HMoscow,HKLP}]
Let $n \geqslant 2$ and $\Lambda \in \Omega_{n}$.  Then
\begin{enumerate}
\item[(1)] The nontrivial elements of $H$ that have fixed points in $C_{k}(\Lambda)$ are $a_1, \ldots, a_n, a_{n+1}:=(a_1 \cdots \, a_n)^{-1}$ and their powers. Moreover, each fixed point of a non-trivial power of $a_{j}$ is also a fixed point of $a_{j}$.

\item[(2)] The group $H$ is the unique generalized Fermat group of $C_{k}(\Lambda)$, that is, if $H' \leqslant {\rm Aut}(C_{k}(\Lambda))$ is a generalized Fermat group of some type $(k',n')$, then $(k',n')=(k,n)$ and $H'=H$. 
\end{enumerate}
 \end{theo*}

\begin{rema}If $\Lambda=(q_{4},\ldots,q_{n+1}) \in \Omega_{n}$, and $A$ is the 
$\mbox{M\"{o}b}(\mathbb{C})$-stabilizer of $\mathscr{B}_{\Lambda}=\{\infty, 0, 1, q_{4},\ldots, q_{n+1}\}$, then 
the uniqueness of $H$ asserts that it is a normal subgroup of $\mbox{Aut}(C_{k}(\Lambda))$
 and therefore $ \mbox{Aut}(C_{k}(\Lambda))/H$ is a spherical group. In other words, there is a short exact sequence of groups
$1\longrightarrow H \longrightarrow \mbox{Aut}(C_{k}(\Lambda))\stackrel{\theta}{\longrightarrow} A \longrightarrow 1$.
 \end{rema}

\subsection{Geometric automorphisms of $H$}\label{geoaut}
Let $\Lambda=(q_4, \ldots, q_{n+1}) \in \Omega_{n}$. We proceed to provide a concrete description of ${\rm Aut}_{g}(H)$. We use the following notations.

\begin{enumerate}
\item  $\mbox{Hom}^{+}(C_{k}(\Lambda))$ is the group of orientation-preserving homeomorphisms of $C_{k}(\Lambda).$
\item $\mbox{Hom}^{+}_H(C_{k}(\Lambda))$ is the normalizer of $H$ in $\mbox{Hom}^{+}(C_{k}(\Lambda))$. 
\item $\mbox{Hom}^{+}(\mathcal{O}(\Lambda))$ is the group of orientation-preserving homeomorphisms of $\bar{\mathbb{C}}$ that preserve the set $\mathscr{B}_{\Lambda}.$
\end{enumerate}

Observe that there is a natural group homomorphism
$$\eta: \mbox{Hom}^{+}_H(C_{k}(\Lambda)) \to \mbox{Hom}^{+}(\mathcal{O}(\Lambda)) \mbox{ such that }\eta(f)\circ \pi = \pi \circ f \mbox{ for each }f \in \mbox{Hom}^{+}_H(C_{k}(\Lambda)).$$The fact that the covering map  \eqref{pi} is characteristic implies that $\eta$ is  surjective. Note that $\ker(\eta)=H$.  Besides, each $f \in \mbox{Hom}^{+}_H(C_{k}(\Lambda))$ defines the geometric automorphism $\Phi_f$ of $H$ given by $\Phi_f(a_j)=f \circ a_j \circ f^{-1}.$

 Let us consider the natural group epimorphism $$\rho: \mbox{Hom}^{+}_H(C_{k}(\Lambda)) \to \mbox{Aut}_g(H) \mbox{ given by }\rho(f)=\Phi_f.$$

\begin{rema} $\mbox{Aut}_g(H)$ does not depend on $\Lambda$. In fact, for each $\hat{\Lambda}=(\hat{q}_4, \ldots, \hat{q}_{n+1}) \in \Omega_n$ we can consider  an orientation-preserving homeomorphism  $\hat{g}: \bar{\mathbb{C}} \to \bar{\mathbb{C}}$ such that \begin{equation}\label{cond}\hat{g}(\infty)=\infty, \hat{g}(0)=0, \hat{g}(1)=1 \mbox{ and } \hat{g}(q_j)=\hat{q}_j \mbox{ for } j=4, \ldots, n+1.\end{equation} 
The fact that \eqref{pi} is characteristic guarantees that $\hat{g}$ lifts to an orientation-preserving homeomorphism $g: C_{k}(\Lambda) \to C_{k}(\hat{\Lambda})$ such that  $gHg^{-1}=H$. Now, the conditions \eqref{cond} imply that, if $\hat{f} \in \mbox{Hom}^{+}_H(C_{k}(\hat{\Lambda}))$ then $$\Phi_{g^{-1}\hat{f}g}(a_j)=g^{-1}\hat{f}g \circ a_j \circ (g^{-1}\hat{f}g)^{-1}=\hat{f} \circ a_j \circ \hat{f}^{-1}=\Phi_{\hat{f}}(a_j).$$ The claim follows after noting that each $f \in \mbox{Hom}^{+}_H(C_{k}(\Lambda))$ is of the form $g^{-1}\hat{f}g.$ 
\end{rema}

\begin{prop}\label{tp} Each $\Phi \in \mbox{Aut}_g(H)$
yields a uniquely determined permutation $\sigma_{\Phi} \in \mathbf{S}_{n+1},$  and the correspondence \begin{equation*}\label{perm}\mbox{Aut}_g(H) \to \mathbf{S}_{n+1} \, \, \mbox{ given by } \,\, \Phi \mapsto \sigma_{\Phi}\end{equation*}is a group isomorphism. In particular, $\mbox{Aut}_g(H)=\langle \Phi_1, \Phi_2\rangle$ where 
$$\Phi_1(a_1, a_2, a_3, \ldots, a_n)=(a_2, a_1, a_3, \ldots, a_n)  \mbox{ and } \Phi_2(a_1, a_2, \ldots, a_{n-1}, a_n)=(a_2, a_3, \ldots,  a_n, a_1).$$
\end{prop}
\begin{proof} Consider the natural group epimorphism $$\hat{\rho}: \mbox{Hom}^{+}(\mathcal{O}(\Lambda)) \to \mathbf{S}_{n+1} \mbox{ given by } \hat{f} \mapsto \hat{\rho}(\hat{f}) \mbox{ where } \hat{f}(q_j)=q_{\hat{\rho}(\hat{f})(j)}.$$

The kernel of $\hat{\rho}$ is, up to isotopy rel ${\mathcal B}_{\Lambda}$, the group generated by the Dehn-twists along simple closed curves contained in $\bar{\mathbb{C}} - \mathscr{B}_{\Lambda}.$ 
We note that, if $f \in \mbox{Hom}^{+}_H(C_{k}(\Lambda))$ then $$\Phi_f(a_j)=f \circ a_j \circ f^{-1}=a_{\hat{\rho}(\eta(f))(j)} \mbox{ for each } j=1, \ldots, n.$$
Thus, there is a group isomorphism $\hat{\eta}:  \mbox{Aut}_g(H) \to \mathbf{S}_{n+1}$ which makes the following diagram commutative.
{\small
$$\begin{tikzpicture}[node distance=3.5 cm, auto]\node (P) {$\mbox{Hom}^{+}_H(C_{k}(\Lambda))$};
\node (B) [right of=P] {$\mbox{Aut}_g(H)$};
\node (A) [below of=P, node distance=1.5 cm] {$\mbox{Hom}^{+}(\mathcal{O}(\Lambda))$};
\node (C) [below of=B, node distance=1.5 cm] {$\mathbf{S}_{n+1}$};
\draw[->] (P) to node {$\rho$} (B);
\draw[->] (P) to node [swap]   {$\eta$} (A);
\draw[->] (A) to node  {$\hat{\rho}$} (C);
\draw[->] (B) to node [swap] {$\hat{\eta}$} (C);
\end{tikzpicture}$$
}
This proves the first statement; the latter one follows directly from the former. 
\end{proof}

\begin{rema}\mbox{}
\begin{enumerate}
\item[(1)]
Let $L \leqslant \mbox{Hom}^{+}(\mathcal{O}(\Lambda))$ and set ${\Q}:=\eta^{-1}(L)$. As $\rho(\Q)={\hat{\eta}^{-1}(\hat{\rho}(L))}$,
the subgroup $\Q^*:=\rho(\Q)=\{\Phi_{T}: T \in \Q\}$ of $\mbox{Aut}_g(H)$ is completely determined by the permutational action $\hat{\rho}(L)$.

\item[(2)]
If $\mbox{Aut}(\mathcal{O}(\Lambda)) \leqslant \mbox{M\"{o}b}(\mathbb{C})$ denotes 
the subgroup of $\mbox{Hom}^{+}(\mathcal{O}(\Lambda))$
consisting of the automorphisms of the orbifold $C_{k}(\Lambda)/H$, then we have the following commutative diagram
{\small
$$\begin{tikzpicture}[node distance=3.5 cm, auto]
\node (P) {$\mbox{Aut}(C_{k}(\Lambda))$};
\node (B) [right of=P] {$\mbox{Aut}_g(H)$};
\node (A) [below of=P, node distance=1.5 cm] {$\mbox{Aut}(\mathcal{O}(\Lambda))$};
\node (C) [below of=B, node distance=1.5 cm] {$\mathbf{S}_{n+1}$};
\draw[->] (P) to node {\tiny $\rho$} (B);
\draw[->] (P) to node   {\tiny $\eta$} (A);
\draw[->] (A) to node  {\tiny $\hat{\rho}$} (C);
\draw[->] (B) to node [swap] {\tiny $\hat{\eta}$} (C);
\end{tikzpicture}$$
}
where neither $\rho$ nor $\hat{\rho}$ are surjective for $n \geqslant 3,$ and that $\mbox{ker}(\rho)\cap \mbox{Aut}(C_{k}(\Lambda))=H.$
\end{enumerate}
\end{rema}

\section{$\mathbb{Z}_{k}^{m}$-actions of signature  $(0;k^{n+1})$}\label{state}
In this section, $k,n \geqslant 2$ and $1 \leqslant m \leqslant n$ are integers such that $(n-1)(k-1)>2$. We recall the following definition.

\begin{defi}
A pair $(S, N)$ is called a {\it $\mathbb{Z}_{k}^{m}$-action of signature  $(0;k^{n+1})$} if 
$S$ is a compact Riemann surface endowed with a group of automorphisms  $N \cong {\mathbb Z}_{k}^{m}$ such that $S/N$ has signature $(0;k^{n+1})$.
\end{defi}

Observe that the Riemann surfaces $S$, such that $(S,N)$ is a $\mathbb{Z}_{k}^{m}$-action of signature  $(0;k^{n+1})$, 
form a complex $(n-2)$-dimensional family in moduli space $\mathscr{M}_g,$ where  \begin{equation}\label{gen}g=1+\tfrac{1}{2}k^{m-1}[(n-1)(k-1)-2].
\end{equation}

\begin{rema} Every ${\mathbb Z}_{k}^{n}$-action of signature  $(0;k^{n+1})$ 
is biholomorphic to a generalized Fermat pair of type $(k,n)$. In particular, all of them are topologically equivalent. On the other extreme, every ${\mathbb Z}_{k}$-action $(S,N)$  of signature  $(0;k^{n+1})$ corresponds to a cyclic $k$-gonal curve $$S \cong \left\{y^{k}=\prod_{j=1}^{n+1}(x-q_{j})^{l_{j}}\right\} \mbox{ and } N=\langle (x,y) \mapsto (x,\omega_{k} y)\rangle,$$ where the integers $l_1, \ldots, l_{n+1}$ lie in $\{1,\ldots,k-1\}$ and are coprime to $k$, and $l_{1}+\cdots+l_{n+1} \equiv 0 \mbox{ mod }k$.
\end{rema}

By the remark above, hereafter we shall only consider the case $n \geqslant 3$ and $1 \leqslant m \leqslant n-1$.

\subsection{Fiber product description for ${\bf m \geqslant 2}$}\label{fibradoproducto}
Let $(S,N)$ be a $\mathbb{Z}_{k}^{m}$-action of signature  $(0;k^{n+1})$, where $n \geqslant 3$, $k \geqslant 2$, $(n-1)(k-1)>2$ and $2 \leqslant m \leqslant n-1$.
Let $\phi_1, \ldots, \phi_m$ be automorphisms of $S$ that generate $N,$ and let $\pi : S \to \bar{\mathbb{C}}$ be a branched regular covering map with deck group $N.$ Set $$N_1=\langle \phi_2, \ldots, \phi_{m}\rangle, \, N_m=\langle \phi_1, \ldots, \phi_{m-1}\rangle \mbox{ and }N_i=\langle \phi_1, \ldots, \phi_{i-1}, \phi_{i+1}, \ldots, \phi_m \rangle,$$for each $i =2, \ldots, m-1.$ Note that $N_{j} \cong {\mathbb Z}_{k}^{m-1}$. We denote by $S_i$ the compact Riemann surface underlying to the quotient $S/N_i$ for each $i$, and by $\pi^i : S \to S_i$  a branched regular covering map with deck group $N_i.$  Observe that $\phi_{i}$ induces an automorphism $\tau_i$ of $S_i$ of order $k$ such that $S/N \cong S_i/\langle \tau_i\rangle$, for each $i=1, \ldots, m.$
If $\pi_i : S_i \to \bar{\mathbb{C}}$ is a branched regular covering map with deck group $\langle \tau_i \rangle$ such that $\pi=\pi_i \circ \pi^i$ then, following \cite[Section 3.2]{H23}, we have that $S \mbox{ is isomorphic to the fiber product }\Pi_{i=1}^m (S_i, \pi_i).$

This description allows us to provide an explicit algebraic description of $S$ in terms of the branch values of $\pi_i.$
Later, we will make this description explicit for the case $m=2$.

\subsection{Descriptions in terms of generalized Fermat curves}
Let $(S, N)$ be a $\mathbb{Z}_{k}^m$-action of signature  $(0;k^{n+1})$, where $n \geqslant 3$, $k \geqslant 2$, $(n-1)(k-1)>2$ and $1 \leqslant m \leqslant n-1$.
After considering a suitable M\"{o}bius transformation, we can assume that the cone points of $S/N \cong \bar{\mathbb C}$ are given by the set 
$$\mathscr{B}_{\Lambda}=\{{q}_1=\infty, {q}_{2}=0, {q}_{3}=1, q_4, \ldots, {q}_{n+1}\}, \, \mbox{ where }\Lambda=(q_{4},\ldots,q_{n+1}) \in \Omega_{n}.$$

As discussed in \S\ref{CFG}, the tuple $\Lambda \in \Omega_{n}$ determines the generalized Fermat curve $C_{k}(\Lambda)$ of type $(k,n)$, and its generalized Fermat group is
$H:=\langle a_1, \ldots, a_n\rangle \cong \mathbb{Z}_{k}^{n}$
where $a_j$ is as \eqref{lin}. In this case, we consider the set $$\mathcal{F}(k, n,m):={\mathcal F}_{{\mathbb Z}_{k}^{m}}(0;k,\stackrel{n+1}{\ldots},k)=\{ K \leqslant H: H/K \cong \mathbb{Z}_{k}^{m}, \,\, \langle a_{j} \rangle \cap K=\{1\}, \,\, j=1,\ldots n+1\}.$$

\begin{rema}
The condition $\langle a_{j} \rangle \cap K=\{1\}$ asserts that eack $K \in \mathcal{F}(k, n,m)$ acts freely. If, in addition, $k$ is prime then the condition $\langle a_{j} \rangle \cap K=\{1\}$ is equivalent to
$a_j \notin K$.
\end{rema}

Theorem \ref{teoremaiso1} asserts that the ${\mathbb Z}_{k}^{m}$-actions of signature  $(0;k^{n+1})$ are parametrized by ${\mathcal F}(k,n,m)$.


\begin{prop}\label{p11}
Let $n \geqslant 3$, $k \geqslant 2$ and $1 \leqslant m \leqslant n-1$ be integers such that $(n-1)(k-1)>2$. If $(S,N)$ is a $\mathbb{Z}_{k}^m$-action of signature  $(0;k^{n+1})$, then there exist  $K_S \in {\mathcal F}(k,n,m)$ and a tuple $({q}_4, \ldots, {q}_{n+1}) \in \Omega_n$
such that $$(S, N) \mbox{ and } (C_{k}({q}_4, \ldots, {q}_{n+1})/K_S, H/K_S) \mbox{ are biholomorphically equivalent.}$$  
\end{prop}

The following result describes the members of ${\mathcal F}(k,n,m)$.

\begin{theo}\label{clasificageneral}
Let $k \geqslant 2$, $n \geqslant 3$ and $1 \leqslant m \leqslant n-1$ be integers such that $(n-1)(k-1)>2.$ Then 
 $K \in {\mathcal F}(k,n,m)$ if and only if 
 \begin{enumerate}
 \item there are integers $1=k_{1}<k_{2}<\cdots<k_{m} \leqslant n$;
 \item there are integers   
$l_{j,i} \in \{0,1,\ldots,k-1\}$, $j=m+1,\ldots,n+1$ and $i=1,\ldots,m$,  satisfying
\begin{enumerate}
\item $(l_{j,1},\ldots,l_{j,m}) \neq (0,\ldots,0)$ for every $j=m+1,\ldots,n+1$, and 
\item $1+l_{m+1,i}+\cdots+l_{n+1,i} \equiv 0 \mod k$, for every $i=1,\ldots,m$, 
\end{enumerate}
\end{enumerate}
such that, if $\{u_{1}, u_{2}, \ldots, u_{n-m} \} = \{1,\ldots,n\} \setminus \{k_{1},\ldots,k_{m}\}$, where $u_{1}<\cdots<u_{n-m}$, then 
$$   
K=\langle  a_{k_{1}}^{l_{m+1,1}}\cdots \, a_{k_{m}}^{l_{m+1},m} a_{u_{1}}^{-1}, \ldots, a_{k_{1}}^{l_{n,1}}\cdots \, a_{k_{m}}^{l_{n},m} a_{u_{n-m}}^{-1}\rangle \cong {\mathbb Z}_{k}^{n-m}.
$$
\end{theo}
\begin{proof}
Each group $K \in {\mathcal F}(k,n,m)$ is the kernel of some group epimorphism $\theta:H \to {\mathbb Z}_{k}^{m}$ with the property that, for every $j=1,\ldots,n+1$, the element $\theta(a_{j})$ has order $k$. Since $a_{1}\cdots \, a_{n+1}=1$, the surjectivity ensures that there integers $1=k_{1}<k_{2}<\cdots<k_{m} \leqslant n$ such that 
$\phi_{1}:=\theta(a_{k_{1}}), \ldots, \phi_{m}:=\theta(a_{k_{m}})$ form a set of generators of ${\mathbb Z}_{k}^{m}$. Thus, 
if $\{u_{1}, u_{2}, \ldots, u_{n-m} \} = \{1,\ldots,n\} \setminus \{k_{1},\ldots,k_{m}\}$, where $u_{1}<\cdots<u_{n-m}$, then 
$\theta(a_{u_{j}})=\phi_{1}^{l_{j,1}}  \cdots \, \phi_{m}^{l_{j,m}}$,  for some $l_{j,i} \in \{0,1,\ldots,k-1\}$.
The condition that $\langle a_{j} \rangle \cap K=\{1\}$ asserts that $(l_{j,1},\ldots,l_{j,m}) \neq (0,\ldots,0)$ for each $j$. Besides, the fact that the product $a_{1}  \cdots \, a_{n+1}$ is trivial implies that 
$1+l_{m+1,i}+\cdots+l_{n+1,i} \equiv 0 \mod k$, for each $i=1,\ldots,m$. In this way, the kernel $K$ of $\theta$ is generated by the $n-m+1$ elements
$$
a_{k_{1}}^{l_{m+1,1}}\cdots \, a_{k_{m}}^{l_{m+1},m} a_{u_{1}}^{-1}, \ldots, a_{k_{1}}^{l_{n,1}}\cdots \, a_{k_{m}}^{l_{n},m} a_{u_{n-m}}^{-1},
a_{k_{1}}^{l_{n+1,1}}\cdots \, a_{k_{m}}^{l_{n+1},m} a_{n+1}^{-1}.
$$The equality 
$$
a_{k_{1}}^{l_{n+1,1}}\cdots \, a_{k_{m}}^{l_{n+1},m} a_{n+1}^{-1}=
a_{k_{1}}^{l_{n+1,1}}\cdots \, a_{k_{m}}^{l_{n+1},m} (a_{1} \cdots \, a_{n})=
a_{k_{1}}^{1+l_{n+1,1}}\cdots \, a_{k_{m}}^{1+l_{n+1},m} (a_{u_{1}} \cdots \, a_{u_{n-m}})=$$
$$
a_{k_{1}}^{-(l_{m+1,1}+\cdots+l_{n,1})}\cdots \, a_{k_{m}}^{-(l_{m+1,m}+\cdots+l_{n,m})} (a_{u_{1}} \cdots \, a_{u_{n-m}})=
(a_{k_{1}}^{l_{m+1,1}}\cdots \, a_{k_{m}}^{l_{m+1},m} a_{u_{1}}^{-1})^{-1} \cdots (a_{k_{1}}^{l_{n,1}}\cdots \, a_{k_{m}}^{l_{n},m} a_{u_{n-m}}^{-1})^{-1},
$$
shows that $K$ is generated by the $n-m$ elements
$
a_{k_{1}}^{l_{m+1,1}}\cdots \, a_{k_{m}}^{l_{m+1},m} a_{u_{1}}^{-1}, \ldots, a_{k_{1}}^{l_{n,1}}\cdots \, a_{k_{m}}^{l_{n},m} a_{u_{n-m}}^{-1},
$ as claimed.
\end{proof}

As a consequence of the above result, we observe the following.
\begin{coro}
Let $k \geqslant 2$, $n \geqslant 3$ and $1 \leqslant m \leqslant n-1$ be integers such that $(n-1)(k-1)>2.$ If
$K \in {\mathcal F}(k,n,m)$ then there exists $\Phi \in {\rm Aut}_{g}(H)$ and there are integers   
$l_{j,i} \in \{0,1,\ldots,k-1\}$, $j=m+1,\ldots,n+1$ and $i=1,\ldots,m$,  satisfying
\begin{enumerate}
\item $(l_{j,1},\ldots,l_{j,m}) \neq (0,\ldots,0)$ for every $j=m+1,\ldots,n+1$, and 
\item $1+l_{m+1,i}+\cdots+l_{n+1,i} \equiv 0 \mod k$, for every $i=1,\ldots,m$, 
\end{enumerate}such that 
$$   
\Phi(K)=\langle  a_{1}^{l_{m+1,1}}\cdots \, a_{m}^{l_{m+1},m} a_{m+1}^{-1}, \ldots, a_{1}^{l_{n,1}}\cdots \, a_{m}^{l_{n},m} a_{n}^{-1}\rangle.
$$
\end{coro}

\subsection{Topological classification of $\mathbb{Z}_k^m$-actions}
Let $n \geqslant 3$, $k \geqslant 2$ and $1 \leqslant m \leqslant n-1$ be integers such that $(n-1)(k-1)>2$. Let $(S, N)$ be a $\mathbb{Z}_k^m$-action of signature  $(0;k^{n+1})$, and assume that the cone points of $S/N$ are $\infty, 0, 1, \hat{q}_4, \ldots, \hat{q}_{n+1}$. By Proposition \ref{p11}, there is a subgroup $K_S \in {\mathcal F}(k,n,m)$ such that 
$(S, N)$ and $(C_{k}(\hat{\Lambda})/K_S, H/K_S)$ are biholomorphically equivalent, where
$\hat{\Lambda}=(\hat{q}_4, \ldots, \hat{q}_{n+1}) \in \Omega_{n}$.

As a consequence of Corollary \ref{corolario1}, we obtain the following result. We provide an alternative proof as it will be needed in the proof of Proposition \ref{propo5}.

\begin{prop}\label{uno}
Let $n \geqslant 3$, $k \geqslant 2$ and $1 \leqslant m \leqslant n-1$ be integers such that $(n-1)(k-1)>2$, and fix $\Lambda \in \Omega_n$.
If $(S, N)$ is a $\mathbb{Z}_k^m$-action of signature  $(0;k^{n+1})$, then  there exists $K \in \mathcal{F}(k, n,m)$ such that $$(S, N) \mbox{ and }(S_K=C_{k}(\Lambda)/K, N_K=H/K) \mbox{ are topologically equivalent.}$$
\end{prop}
\begin{proof}
We fix $\Lambda=({q}_4, \ldots, {q}_{n+1}) \in \Omega_n$
and consider  an orientation-preserving homeomorphism  $\hat{f}: \bar{\mathbb{C}} \to \bar{\mathbb{C}}$ such that 
$\hat{f}(\infty)=\infty, \hat{f}(0)=0, \hat{f}(1)=1,$ and $\hat{f}(\hat{q}_j)={q}_{j},$ for $j=4, \ldots, n+1.$
It follows that there is an orientation-preserving homeomorphism $f: C_{k}(\hat{\Lambda}) \to C_{k}(\Lambda)$ such that $\hat{f} \circ \pi=\pi \circ f$. In particular,  one has that $fHf^{-1}=H$. The way as we have chosen $\hat{f}$ asserts that $fK_Sf^{-1}=K_S$. Now, if we define $S_{K_S}:=C_{k}(\Lambda)/K_{S}$ and $N_{K_S}:=H/K_{S},$ then $f$ induces an orientation-preserving homeomorphism $g: S_{K_S} \to S$ such that the following diagram commutes.
{\small
$$\begin{tikzpicture}[node distance=3.5 cm, auto]
\node (P) {$C_{k}(\Lambda)$};
\node (B) [right of=P] {$C_{k}(\hat{\Lambda})$};
\node (A) [below of=P, node distance=1.5 cm] {$S_{K_S}$};
\node (C) [below of=B, node distance=1.5 cm] {$S$};
\node (D) [below of=A, node distance=1.5 cm] {$\bar{\mathbb{C}}$};
\node (E) [below of=C, node distance=1.5 cm] {$\bar{\mathbb{C}}$};
\draw[->] (P) to node {\tiny $f$} (B);
\draw[->] (A) to node {\tiny $N_{K_S}$} (D);
\draw[->] (C) to node[swap]  {\tiny $N$} (E);
\draw[->] (D) to node {\tiny $\hat{f}$} (E);
\draw[->] (P) to node   {\tiny $K_S$} (A);
\draw[->] (A) to node  {\tiny $g$} (C);
\draw[->] (B) to node [swap] {\tiny $K_S$} (C);
\draw[->, bend left] (B) to node {\tiny $H$} (E);
\draw[->, bend right] (P) to node [swap] {\tiny $H$} (D);
\end{tikzpicture}$$
}
Note that $gN_{{K_S}}g^{-1}=N.$ 
\end{proof}

As a consequence of Theorem \ref{abeliano2}, we observe the following result. 


\begin{theo}\label{teorema2}
Let $n \geqslant 3$, $k \geqslant 2$ and $1 \leqslant m \leqslant n-1$ be integers such that $(n-1)(k-1)>2$. Let 
 $K_1, K_2 \in \mathcal{F}(k, n,m).$ The pairs $(S_{K_1}, N_{K_1})$ and $(S_{K_2}, N_{K_2})$ are topologically equivalent if and only if
there exists $\Phi \in \mbox{Aut}_g(H)$ such that $\Phi(K_1)=K_2.$
\end{theo}


Observe that there is a natural action $$\mbox{Aut}_g(H) \times \mathcal{F}(k, n,m) \to \mathcal{F}(k, n,m) \mbox{ given by } (\Phi, K) \mapsto \Phi(K).$$
In our situation, Corollary \ref{abeliano3} reads as follows.

\begin{coro}\label{felix}
The number of pairwise topologically inequivalent $\mathbb{Z}_k^m$-actions with signature  $(0;k^{n+1})$ is equal to the 
cardinality of the quotient set $\mathcal{F}(k, n,m)/\mbox{Aut}_g(H)$.
\end{coro}

\begin{example}[The classical case $m=1$]\label{casom=1}
The case $m=1$ and $k=2$ corresponds to the hyperelliptic involution, which is known to be unique. So, let us restrict to the case $k\geqslant 3$, and $n \geqslant 2$ such that $(n-1)(k-1)>2$. 
If we set $$\Delta=\{(l_{2},\ldots,l_{n}): 1 \leqslant l_{2},\ldots,l_{n} \leqslant k-1, \; \gcd(k,l_{2})=\cdots=\gcd(k,l_{n})=\gcd(k,1+l_{2}+\cdots+l_{n})=1\},$$
then the elements of ${\mathcal F}(k,n,1)$ are those of the form
$K(l_{2},\ldots,l_{n})=\langle a_{1}^{l_{2}}a_{2}^{-1},\ldots,a_{1}^{l_{n}}a_{n}^{-1}\rangle$, where $(l_{2},\ldots,l_{n}) \in \Delta$.

The Riemann surface $S_{K(l_{2},\ldots,l_{n})}=C_{k}(\Lambda)/K(l_{2},\ldots,l_{n})$ is algebraically given by
$$y^{k}=x^{l_{2}}(x-1)^{l_{3}}\prod_{j=1}^{n-2}(x-\lambda_{j})^{l_{3+j}},$$
where $\Lambda=(\lambda_{1},\ldots,\lambda_{n-2}) \in \Omega_{n}$, and $l_{n+1} \in \{1,\ldots,k-1\}$ is congruent to $-(1+l_{2}+\cdots+l_{n})$ modulo $k$.

If $l \in \{1,\ldots,k-1\}$ satisfies that $\gcd(k,l)=1$, then we denote by $l^{-1} \in \{1,\ldots,k-1\}$ the element such that $l l^{-1} \equiv 1 \mbox{ mod }k$. The actions of $\Phi_{1}$ and $\Phi_{2}$ on these groups are as follows:
$$\Phi_{1}(K(l_{2},\ldots,l_{n}))=K(l_{2}^{-1},l_{2}^{-1}l_{3},\ldots,l_{2}^{-1}l_{n}),$$
$$\Phi_{2}(K(l_{2},\ldots,l_{n}))=K(q,ql_{2},ql_{3},\ldots,ql_{n-1}), \; q=k-(1+l_{2}+\cdots+l_{n})^{-1}.$$

As a consequence of Corollary \ref{uno}, the number of topologically different ${\mathbb Z}_{k}$-actions with signature $(0;k,\stackrel{n+1}{\ldots},k)$ is equal to the cardinality of $\Delta/\langle B_{1},B_{2}\rangle$, where
$$B_{1}(l_{2},\ldots,l_{n})=(l_{2}^{-1},l_{2}^{-1}l_{3},\ldots,l_{2}^{-1}l_{n}), \;
B_{2}(l_{2},\ldots,l_{n})=(q,ql_{2},ql_{3},\ldots,ql_{n-1}), \; q=k-(1+l_{2}+\cdots+l_{n})^{-1}.$$
\end{example}

%
%
%
%
%

\subsection{Topological actions and extra automorphisms}\label{state2}
Let $n \geqslant 3$, $k \geqslant 2$ and $2 \leqslant m \leqslant n-1$ be integers such that $(n-1)(k-1)>2$. We consider triples $(S,N, G)$ where $(S, N)$ is a $\mathbb{Z}_k^m$-action of signature 
 $(0;k^{n+1})$, and $S$ is endowed with a group of automorphisms $G$ such that $N \trianglelefteq G \leqslant \mbox{Aut}(S).$ Consider $\Lambda=(q_4, \ldots, q_{n+1}) \in \Omega_{n}$  such that $$S \cong C_{k}(\Lambda)/K_S \, \mbox{ and } \, N \cong H/K_S \mbox{ for some }K_S \in \mathcal{F}(k,n, m).$$We recall that  there is a short exact sequence of groups $1\longrightarrow H \longrightarrow \mbox{Aut}(C_{k}(\Lambda))\stackrel{\theta}{\longrightarrow} A\longrightarrow 1,$ where $A$ is the $\mbox{M\"{o}b}(\mathbb{C})$-stabilizer of $\mathscr{B}_{\Lambda}=\{\infty, 0, 1, q_{4},\ldots, q_{n+1}\}.$ As $A$ has a subgroup $L$ isomorphic to $G/N$, there is an induced short exact sequence of groups\begin{equation*}\label{short7} 1\longrightarrow H \longrightarrow Q_{S, G}:=\theta^{-1}(L)\stackrel{\theta}{\longrightarrow} L\longrightarrow 1.\end{equation*}Observe that $H \trianglelefteq Q_{S, G}$ and $Q_{S, G}/H \cong G/N$, $K_S \trianglelefteq {Q_{S,G}}$ and $Q_{S,G}/K_S \cong G$, and $S/G \cong C_{k}(\Lambda)/{Q_{S,G}}.$

All the above can be summarized in the following commutative diagram, where $C=C_{k}(\Lambda).$ 
{\small
$$\begin{tikzpicture}[node distance=5 cm, auto]
\node (P) {$S$};
\node (B) [right of=P, node distance=4 cm] {$C$};
\node (A) [below right of=P, node distance=2.5 cm] {$S/N=C/H$};
\node (C) [below of=A, node distance=2 cm] {$\,\,\,\,\,\,\,S/G=C/Q_{S,G}$};
\draw[->] (P) to node [swap]{\tiny $N$} (A);
\draw[->] (B) to node {\tiny $K_S$} (P);
\draw[->] (B) to node {\tiny $H$} (A);
\draw[->] (A) to node  {\tiny $G/N$} (C);
\draw[->, bend right] (P) to node  [swap] {\tiny $G$} (C);
\draw[->, bend left] (B) to node {\tiny $Q_{S, G}$} (C);
\end{tikzpicture}$$
}
We recall that if $\mathcal{Q}$ is any group of orientation-preserving homeomorphisms of $C_{k}(\Lambda)$ such that $H \triangleleft \mathcal{Q}$, then we may consider the representation (see the proof of Proposition \ref{tp})
\begin{equation*}\label{repre}\rho_{Q}: \Q \to \mbox{Aut}_g(H) \mbox{ given by }\rho_{Q}(f)=\Phi_f.\end{equation*}
We denote its image by $\Q^*.$ In particular, as $H \triangleleft Q_{S,G}$, we may consider 
the group epimorphism  \begin{equation}\label{repre222}\rho_{Q_{S,G}}: Q_{S,G} \to \mbox{Aut}_{g}(H) \mbox{ given by }\rho_{Q_{S,G}}(f)=\Phi_f.\end{equation}
 As the kernel of $\rho_{Q_{S,G}}$ is $H,$ we have that  $Q_{S,G}/H\cong Q_{S,G}^*$. Note that $\rho_{Q_{S,G}}(f)$ is uniquely determined by the permutation induced by  $\theta(f)$. We define $$\mathscr{C}_k({Q_{S,G}}):=\{ K \in \mathcal{F}(k,n,m) :   K \mbox{ is } Q_{S,G}^*\mbox{-invariant}\}.$$

Note that $\mathscr{C}_k({Q_{S,G}})$ is nonempty as $K_S$ belongs to it. If $K \in \mathscr{C}_k({Q_{S, G}})$ then $$(S_K=C_{k}(\Lambda)/K, N_K=H/K)$$is a $\mathbb{Z}_k^m$-action of signature  $(0;k^{n+1})$ such that there exists $\Q \leqslant \mbox{Hom}_{H}^+(C_{k}(\Lambda))$ 
which contains $H$ and $K$ as normal subgroups, satisfying $\Q^*=Q_{S, G}^*$ and that $$N_K \trianglelefteq G_K = \Q/K \leqslant \mbox{Hom}^+(S_K) \mbox { and } G_K/N_K \cong \Q/H \leqslant \mbox{Hom}^+(\bar{\mathbb{C}}).$$In addition, $S_K/G_K$ and $C_{k}(\Lambda)/\Q$ are equivalent, as topological orbifolds.

\begin{prop}\label{propo5}
Let $n \geqslant 3$, $k \geqslant 2$ and $1 \leqslant m \leqslant n-1$ be integers such that $(n-1)(k-1)>2$. Let $(S, N, G)$ be a triple such that  $(S, N)$ is a $\mathbb{Z}_k^m$-action of signature 
 $(0;k^{n+1})$ and $S$ admits a group of automorphisms $G$ such that $N \trianglelefteq G \leqslant \mbox{Aut}(S).$
Let $C_{k}(\Lambda)$ and $Q_{S,G}^* \leqslant \mbox{Aut}_g(H)$ be as before. Then, up to topological equivalence, each triple $(\hat{S}, \hat{N}, \hat{G})$ such that $(\hat{S}, \hat{N})$ is a $\mathbb{Z}_k^m$-action of signature  $(0;k^{n+1})$ and $\hat{S}$ admits a group of automorphisms $\hat{G}$ with $\hat{N} \trianglelefteq \hat{G} \leqslant \mbox{Aut}(\hat{S})$ and satisfying that $Q_{\hat{S}, \hat{G}}^*=Q_{S,G}^*$ corresponds to a member of $\mathscr{C}_k(Q_{S,G}).$
\end{prop}
\begin{proof} 
Let $(\hat{S}, \hat{N}, \hat{G})$ be a triple as in the statement of the proposition, and let $\infty, 0, 1,$ $\hat{q}_4, \ldots, \hat{q}_{n+1}$ be the cone points of $\hat{S}/\hat{N},$ so $\hat{\Lambda}=(\hat{q}_{4},\ldots,\hat{q}_{n+1}) \in \Omega_{n}$. By Proposition \ref{uno}, there exists $K_{\hat{S}} \in \mathcal{F}(k, n,m)$ such that 
$(\hat{S}, \hat{N})$ and $(S_{K_{\hat{S}}}=C_{k}(\Lambda)/K_{\hat{S}}, N_{K_{\hat{S}}}=H/K_{\hat{S}})$ are topologically equivalent. In addition, there is a group ${Q}_{\hat{S}, \hat{G}}\leqslant \mbox{Aut}(C_{k}(\hat{\Lambda}))$ such that $H, K_{\hat{S}} \trianglelefteq \hat{Q}_{\hat{S}, \hat{G}}$ and  $\hat{G}\cong\hat{Q}_{\hat{S}, \hat{G}}/K_{\hat{S}}$. 
All the above is summarized in the following commutative diagram, where $f,g$ and $\hat{f}$ are as in the proof of Proposition \ref{uno}.
{\small
$$
\begin{tikzpicture}[node distance=4.5 cm, auto]
\node (P) {$C_{k}(\Lambda)$};
\node (B) [right of=P] {$C_{k}(\hat{\Lambda})$};
\node (A) [below of=P, node distance=2 cm] {$S_{K_{\hat{S}}}$};
\node (C) [below of=B, node distance=2 cm] {$\hat{S}$};
\node (D) [below of=A, node distance=2 cm] {$\bar{\mathbb{C}}$};
\node (E) [below of=C, node distance=2 cm] {$\bar{\mathbb{C}}$};
\node (F) [below of=E, node distance=2 cm] {$\bar{\mathbb{C}}$};
\draw[->] (E) to node[swap] {\tiny $\hat{G}/\hat{N}$} (F);
\draw[->, bend left] (C) to node {\tiny $\hat{G}$} (F);
\draw[->, bend left=40] (B) to node {\tiny ${Q}_{\hat{S}, \hat{G}}$} (F);    
\draw[->] (P) to node {\tiny $f$} (B);
\draw[->] (A) to node {\tiny $N_{K_{\hat{S}}}=g^{-1}\hat{N}g$} (D);
\draw[->] (C) to node[swap]  {\tiny $\hat{N}$} (E);
\draw[->] (D) to node {\tiny $\hat{f}$} (E);
\draw[->] (P) to node   {\tiny $K_{\hat{S}}=f^{-1}K_{\hat{S}}f$} (A);
\draw[->] (A) to node  {\tiny $g$} (C);
\draw[->] (B) to node [swap] {\tiny $K_{\hat{S}}$} (C);
\draw[->, bend left] (B) to node {\tiny $H$} (E);
\draw[->, bend right] (P) to node [swap] {\tiny $H$} (D);
\end{tikzpicture}
$$
}

We consider the group of homeomorphisms $g^{-1}\hat{G}g \leqslant \mbox{Hom}^+(S_{K_{\hat{S}}})$. As $Q_{\hat{S}, \hat{G}}^*=Q_{S, G}^*$, we have that $S/G$ and $\hat{S}/\hat{G}$ are isomorphic as topological orbifolds. It follows that $S/G$ and $S_{{K}_{\hat{S}}}/(g^{-1}\hat{G}g)$ are isomorphic as topological orbifolds too. Observe that $N_{K_{\hat{S}}} =g^{-1}\hat{N}g \trianglelefteq g^{-1}\hat{G}g$ and that $(g^{-1}\hat{G}g)/N_{K_{\hat{S}}}=\hat{f}^{-1}(\hat{G}/\hat{N})\hat{f}$. Besides, the group $g^{-1}\hat{G}g$ lifts to a group $\Q \leqslant \mbox{Hom}^+(C_k(q_4, \ldots, q_{n+1}))$ such that $H, K \trianglelefteq \mathcal{Q}$ and $\Q/K \cong g^{-1}\hat{G}g.$ Now, the fact that $S/G$ and $S_{{K}_{\hat{S}}}/g^{-1}\hat{G}g$ are isomorphic as topological orbifolds implies that the action by conjugation of $\Q$ agrees with the one of $Q_{S, G}^*$, showing that $K_{\hat{S}} \in \mathscr{C}_p(Q_{S, G})$ as desired.
\end{proof}

Consider a triple $(S, N, G)$ as above.
We denote by $${\mathcal N}_{Q_{S,G}} =\{\Phi \in {\rm Aut}_{g}(H) :  \Phi Q^{*}_{S,G} \Phi^{-1}=Q^{*}_{S,G}\} \leqslant {\rm Aut}_{g}(H)$$ the normalizer of $Q^{*}_{S,G}$ in ${\rm Aut}_{g}(H) $. We observe that $\mathscr{C}_p(Q_{S,G})$ is ${\mathcal N}_{Q_{S,G}}$-invariant.

If $(\hat{S},\hat{N},\hat{G})$ is another triple as above, such that $\hat{G}/\hat{N}$ induces $Q^{*}_{S,G}$, then 
there exists some $K \in \mathscr{C}_p(Q_{S,G})$ such that there is an orientation-preserving homeomorphism $h:S_{K}=C_{k}(\Lambda)/K \to \hat{S}$ such that $h N_{K} h^{-1}=\hat{N}$ and $h G_{K} h^{-1}=\hat{G}$.
So, following similar arguments as in the proof of Proposition \ref{propo5}, we obtain the following result.

\begin{theo}\label{teorema3}
Let $K_{1},K_{2} \in \mathscr{C}_p(Q_{S,G})$. The triples $(S_{K_{1}},N_{K_{1}},G_{K_{1}})$ and $(S_{K_{2}},N_{K_{2}},G_{K_{2}})$ are topologically equivalent if and only if there exists $\Phi \in {\mathcal N}_{Q_{S,G}}$ such that $\Phi(K_{1})=K_{2}$.
\end{theo}

Analogously to Corollary \ref{felix}, we have the following result.

\begin{coro}\label{coca7}
Let  $(S,N,G)$  be a triple as above. 
Then, the number of pairwise topologically inequivalent triples $(\hat{S},\hat{N},\hat{G})$, where $(\hat{S},\hat{N})$ is a ${\mathbb Z}_{k}^{m}$-action of signature  $(0;k^{n+1})$ and $\hat{N} \triangleleft \hat{G} \leqslant {\rm Aut}(\hat{S})$ is such that $\hat{G}/\hat{N}$ induces $Q^{*}_{S,G}$, is equal to  
the cardinality of the quotient set $\mathscr{C}_k(Q_{S,G})/{\mathcal N}_{Q_{S,G}}$.
\end{coro}

\begin{example}[Case $m=1$ and $G$ a dihedral group]
Let us consider those pairs $(S,G)$, where $G=\mathbf{D}_{k}=\langle r,s: r^{k}=s^{2}=(rs)^{2}=1\rangle$, and $S/G$ has signature $(0;2,2,k,\stackrel{n}{\ldots},k)$, where $k \geqslant 3$ is odd, and $n \geqslant 2$.
As a consequence of the Riemann-Hurwitz formula (and the fact that $k$ is odd), $S/\langle r \rangle$ has signature $(0;k,\stackrel{2n}{\ldots},k)$, and it admits a conformal involution permuting the $2n$ cone points in pairs. This involution induces the order two cyclic subgroup ${\mathcal Q}^{*}_{S,G}=\langle \Psi \rangle \leqslant {\rm Aut}_{g}(H)$, where $\Psi(a_{2j-1})=a_{2j}$, $\Psi(a_{2j})=a_{2j-1}$, for $j=1,\ldots,n$.
If we set
$$\Delta=\{(l_{3},l_{5},\ldots,l_{2n-1}): l_{3},l_{5},\ldots,l_{2n-1} \in \{1,\ldots,k-1\}, \;  \gcd(k,l_{3})=\cdots=\gcd(k,l_{2n-1})=1  \},$$
then $\mathscr{C}_{k}(Q_{S,G})$ consists of the groups 
$$K(l_{3},l_{5},\ldots,l_{2n-1})=\langle a_{1}a_{2},a_{1}^{l_{3}}a_{3}^{-1}, a_{1}^{l_{3}}a_{4}, \ldots, a_{1}^{l_{2n-3}}a_{2n-3}^{-1}, a_{1}^{l_{2n-3}}a_{2n-2}, a_{1}^{l_{2n-1}}a_{2n-1}^{-1}\rangle, \; (l_{3},l_{5},\ldots,l_{2n-1}) \in \Delta.$$ 

The Riemann surface $S_{K(l_{3},l_{5},\ldots,l_{2n-1})}$ is algebraically described by one of the form
$$y^{k}=F(x)=x^{k-1}(x-1)^{l_{3}}(x-\lambda_{1})^{k-l_{3}} \prod_{j=2}^{n}(x-\lambda_{2j-1})^{l_{2j-1}}\left(x-\tfrac{\lambda_{1}}{\lambda_{2j-1}} \right)^{k-l_{2j-1}},$$
where $(\lambda_{1},\lambda_{2}=\lambda_{1}/\lambda_{3},\lambda_{3},\lambda_{4}=\lambda_{1}/\lambda_{5},\lambda_{5},\ldots, \lambda_{2n-2}=\lambda_{1}/\lambda_{2n-1},\lambda_{2n-1}) \in \Omega_{2n-1}$, 
and the dihedral group is generated by the automorphisms
$$r(x,y)=(x,\omega_{k}y) \mbox{ and } s(x,y)=\left(\frac{\lambda_{1}}{x}, Q(x)y^{k-1}\right), \mbox{ where }Q(x)^{k}=\dfrac{F(\frac{\lambda_{1}}{x})}{F(x)^{k-1}}.$$

Also, 
${\mathcal N}_{Q_{S,G}}=\langle \Psi_{1}=\Phi_{1},\Psi_{2},\Psi_{3} \rangle$, where:
\begin{enumerate}
\item $\Psi_{1}$ permutes $a_{1}$ with $a_{2}$, and fixes all the others $a_{j}'$s.
\item $\Psi_{2}$ permutes $a_{1}$ with $a_{3}$, permutes $a_{2}$ with $a_{4}$, and  fixes all the others $a_{j}'$s.
\item  $\Psi_{3}(a_{2n-1})=a_{1}$, $\Psi_{3}(a_{2n})=a_{2}$, $\Psi_{3}(a_{2j-1})=a_{2j+1}$ and $\Psi_{3}(a_{2j})=a_{2j+2}$, $j=1,\ldots,n-1$.
\end{enumerate}

If $a,b \in \{1,\ldots,k-1\}$ are such that $\gcd(k,a)=\gcd(k,b)=1$, then we denote by ``$-a$" the element $k-a$, by ``$a^{-1}$" the element in $\{1,\ldots,k-1\}$ such that $aa^{-1} \equiv 1 \mbox{ mod }k$, and by ``$ab$" the element in $\{1,\ldots,k-1\}$ which is the $k$-residue of the product of $a$ and $b$. 
The elements $\Psi_{1}, \Psi_{2}$ and $\Psi_{3}$, induce the following bijections of $\Delta$:
$$A_{1}(l_{3},l_{5},\ldots,l_{2n-1})=(-l_{3},-l_{5},\ldots,-l_{2n-1}), \; 
A_{2}(l_{3},l_{5},\ldots,l_{2n-1})=(-l_{3}^{-1},-l_{3}^{-1}l_{5},\ldots,-l_{3}^{-1}l_{2n-1}),$$
$$A_{3}(l_{3},l_{5},\ldots,l_{2n-1})=(l_{2n-1}^{-1},l_{2n-1}^{-1}l_{3},l_{2n-1}^{-1}l_{5},\ldots,l_{2n-1}^{-1}l_{2n-3}).$$

As a consequence of Corollary \ref{coca7}, we obtain that the number of topologically different actions of $\mathbf{D}_{k}$, $k \geqslant 3$, with quotient signature $(0;2,2,k,\stackrel{n}{\ldots},k)$ is equal to the cardinality of the quotient set $\Delta/\langle A_{1},A_{2},A_{3}\rangle$.
\end{example}

\section{The case $m=2$ and $k=p$ prime}
In this section we restrict to the case of ${\mathbb Z}_{p}^{2}$-actions of signature $(0;p^{n+1})$, where $p \geqslant 2$ is prime, $n \geqslant 3$ and $(n-1)(p-1)>2$.

\subsection{Description of ${\mathcal F}(p,n,2)$}
In this case, as a consequence of Theorem \ref{clasificageneral}, we can describe the elements of ${\mathcal F}(p,n)={\mathcal F}(p,n,2)$.

\begin{theo}\label{max}
If $K \in {\mathcal F}(p,n)$ then one of the following statements holds. 
\begin{enumerate}
\item There are integers $r_{3},s_{3}, \ldots, r_n, s_n \in \{0,1,\ldots,{p}-1\}$ satisfying that $(r_{j},s_{j}) \neq (0,0)$ for each $j$, and  that 
$$(1+r_{3}+\cdots+r_{n},1+s_{3}+\cdots+s_{n}) \nequiv (0,0) \mbox{ mod }p$$ in such a way that $K=\langle a_{1}^{r_{3}}a_{2}^{s_{3}}a_{3}^{-1},\ldots, a_{1}^{r_{n}}a_{2}^{s_{n}}a_{n}^{-1}\rangle.$

\item There is an integer $2 \leqslant t \leqslant n-1$, there are integers $l_{2},\ldots,l_{t} \in \{1,\ldots,{p}-1\}$, and there are integers  $r_{t+2},s_{t+2}, \ldots, r_n, s_n \in \{0,1,\ldots,p-1\}$ satisfying that $(r_{j},s_{j}) \neq (0,0)$ for each $j$, and  that $$(1+l_{2}+\cdots+l_{t}+r_{t+2}+\cdots+r_{n},1+s_{t+2}+\cdots+s_{n}) \nequiv (0,0) \mbox{ mod } p$$ in such a way that $K=\langle a_{1}^{l_{2}}a_{2}^{-1},\ldots,a_{1}^{l_{t}}a_{t}^{-1},a_{1}^{r_{t+2}}a_{t+1}^{s_{t+2}}a_{t+2}^{-1},\ldots, a_{1}^{r_{n}}a_{t+1}^{s_{n}}a_{n}^{-1}\rangle.$
\end{enumerate}
\end{theo}

\begin{proof} The proof follows the same ideas employed in the proof of Theorem \ref{clasificageneral}.
\end{proof}


We say that $K$ is of type (1) or (2) according to the enumeration in the theorem above.

\subsection{Algebraic descriptions}\label{Sec:fibrado}
We recall that, as observed in  \S\ref{fibradoproducto}, each ${\mathbb Z}_{p}^{m}$-action can be described as a fiber product. We proceed to make such a description explicit for the case $m=2$ and $k=p$ prime, and for each possible group $K \in {\mathcal F}(p,n)$. If $\Lambda=(q_{4},\ldots,q_{n+1}) \in \Omega_{n}$, then we consider the ${\mathbb Z}_{p}^{2}$-action 
$$(S_{K}=C_{k}(\Lambda)/K,N_{K}=H/K=\langle \phi_{1},\phi_{2}\rangle),$$
where $\phi_{1}$ and $\phi_{2}$ are as defined in the  proof of Theorem \ref{max}.

For $j=1, 2$, set $S_{j}:=S_{K}/\langle \phi_{j} \rangle$ and denote by $\pi_{j}:S_{j} \to \bar{\mathbb C}$ the branched regular covering map with deck group $\langle \tau_{j}\rangle=N_{K}/\langle \phi_{j} \rangle \cong {\mathbb Z}_{p}$. The branch locus of $\pi_{1}$ is given by the set $\{q_{t+1},\ldots,q_{n+1}\}$ and therefore an algebraic description for $S_{1}$ is given by
$$S_{1}: y_{1}^{{p}}=(x-q_{t+1})\prod_{j=t+2}^{n+1}(x-q_{j})^{s_{j}},$$where $s_{n+1} \equiv -(1+s_{t+2}+\cdots+s_{n}) \mbox{ mod }{p}.$ In this model $\pi_{1}(x,y_{1})=x$. Similarly, the branch locus of $\pi_{2}$ is given by the set $\{q_{1},\ldots,q_{t},q_{t+2},\ldots,q_{n+1}\}$ and therefore an algebraic description for $S_{2}$ is given as follows.  If $K$ is of type (1) then $S_2$ is given by
$$S_2 : y_{2}^{p}=\prod_{j=3}^{n+1}(x-q_{j})^{r_{j}}$$where $r_{n+1} \equiv -(1+r_{3}+\cdots+r_{n}) \mbox{ mod }p.$ If $K$ is of type (2) then
$$S_{2}:y_{2}^{p}=\prod_{i=2}^{t}(x-q_{i})^{l_{i}}\prod_{j=t+2}^{n+1}(x-q_{j})^{r_{j}}$$where $r_{n+1} \equiv -(1+l_{2}+\cdots+l_{t}+r_{t+2}+\cdots+r_{n}) \mbox{ mod }p.$ In this model $\pi_{2}(x,y_{2})=x$.

All the above coupled with the discussion in \S\ref{fibradoproducto} is the proof of the following result.

\begin{prop}\label{pp5}
Let $(S_K, N_K)$ be a ${\mathbb Z}_{p}^{2}$-action of signature $(0;p^{n+1})$. Then, with the same notations as in Theorem \ref{max}, an algebraic description of $S_K$ is given as follows.
\begin{enumerate}
\item If $K$ is of type (1) then $$S_{K}: \left\{ \begin{array}{l}
y_{1}^{p}=x\prod_{j=3}^{n+1}(x-q_{j})^{s_{j}}\\
y_{2}^{p}=\prod_{j=3}^{n+1}(x-q_{j})^{r_{j}}\\
\end{array}
\right.
$$where $s_{n+1} \equiv -(1+s_{t+2}+\cdots+s_{n}) \mbox{ mod }{p}$ and $r_{n+1} \equiv -(1+r_{3}+\cdots+r_{n}) \mbox{ mod }p.$

\item If $K$ is of type (2) then $$S_{K}: \left\{ \begin{array}{l}
y_{1}^{p}=(x-q_{t+1})\prod_{j=t+2}^{n+1}(x-q_{j})^{s_{j}}\\
y_{2}^{p}=\prod_{i=2}^{t}(x-q_{i})^{l_i}\prod_{j=t+2}^{n+1}(x-q_{j})^{r_{j}}\\
\end{array}
\right.
$$where $s_{n+1} \equiv -(1+s_{t+2}+\cdots+s_{n}) \mbox{ mod }{p}$ and $r_{n+1} \equiv -(1+l_{2}+\cdots+l_{t}+r_{t+2}+\cdots+r_{n}) \mbox{ mod }p.$
\end{enumerate}

In both cases (1) and (2), the group $N_{K}$ corresponds to $\langle \phi_{1}(x,y_{1},y_{2})=(x,y_{1},\omega_{p} y_{2}),\phi_{2}(x,y_{1},y_{2})=(x,\omega_{p} y_{1},y_{2})\rangle$.
\end{prop}

\subsection{Jacobian variety}  We recall that the Jacobian variety $JS$ of a compact Riemann surface of genus $g$ is an irreducible principally polarized abelian variety of dimension $g$. By the classical Torelli's theorem, the Jacobian variety $JS$ determines $S$, namely, $S \cong S'$  if and only if $JS \cong JS'$.

For each group of automorphisms $G$ of a Riemann surface $S$, we denote by $S_G$ the underlying Riemann surface structure of the orbifold $S/G$. Let $(X, H_0)$ be a generalized Fermat pair of type $(p,n)$, where $p$ is prime. Following the main result of  \cite{CHQ}, $JX$ decomposes, up to isogeny, as follows:\begin{equation}\label{holaa}JX \sim \prod_{H_r} JX_{H_r},\end{equation}
where $H_r$ runs over all subgroups of $H_{0}$ which are isomorphic to ${\mathbb Z}_{p}^{n-1}$ and such that $X/H_r$ has positive genus. In addition, the cyclic $p$-gonal curves $X_{H_r}$ run over all curves of the form
$$y^{p}=\prod_{j=1}^{r}(x-\mu_{j})^{\alpha_{j}},$$
where $\{\mu_{1},\ldots,\mu_{r}\} \subset \{\infty,0,1,\lambda_{1},\ldots,\lambda_{n-2}\}$, $\mu_{i} \neq \mu_{j}$ if $i \neq j$, and  $\alpha_{j} \in \{1,2,\ldots,p-1\}$ satisfying that:
\begin{enumerate}
\item[(i)] if every $\mu_{j} \neq \infty$, then $\alpha_{1}=1$, $\alpha_{2}+\cdots+\alpha_{r} \equiv -1 \mbox{ mod }p$, and
\item[(ii)] if some $\mu_j=\infty$, then  $\alpha_{1}+\cdots+\alpha_{j-1}+\alpha_{j+1}+\cdots +\alpha_{r} \equiv -1 \mbox{ mod }p$.
\end{enumerate}

Let $p$ be a prime number and let $(S,N)$ be a $\mathbb{Z}_{p}^{m}$-action of signature $(0;p^{n+1})$. As observed in Proposition \ref{p11},  there is a generalized Fermat pair $(X,H_0)$ of type $(p,n)$, and a subgroup $K \cong {\mathbb Z}_{p}^{n-m}$ of $H_0$ such that $S \cong X/K$ and $N \cong H_0/K$. Since $$JX \sim JS \times P(X/S).$$ where $P(X/S)$ is the Prym variety associated to the covering map $X \to S=X/K$, the isogeny decomposition \eqref{holaa}  permits to state the following conjecture.

\begin{conj*}
Let $p \geqslant 2$ be a prime number  and let $(S,N)$ be a $\mathbb{Z}_{p}^{m}$-action of signature  $(0;p^{n+1})$, where $2 \leqslant m \leqslant n-1$. Then   
$$JS \sim \prod_{L \in \mathscr{L}}JS_L \, \mbox{ where } \, \mathscr{L}=\{L \leqslant N : L \cong \mathbb{Z}_p^{m-1}\}.$$
\end{conj*}

The conjecture above does hold for the case $m=2, $ and the proof 
can be obtained as a consequence of a result due to Kani-Rosen in \cite{KR}.

\begin{theo}
Let $p \geqslant 2$ be a prime number and $n \geqslant 3$. Let $(S,N)$ be a $\mathbb{Z}_{p}^{2}$-action of signature  $(0;p^{n+1})$. Then  
$$JS \sim \prod_{L \in \mathscr{L}}JS_L \mbox{ where } \mathscr{L}=\{L \leqslant N : L \cong \mathbb{Z}_p\}.$$
\end{theo}

\begin{proof}
Let us write $N=\langle \phi_1, \phi_2 \rangle.$
We have that$$\mathscr{L}=\{L_1=\langle \phi_1\rangle, L_2=\langle \phi_2\rangle, L_3=\langle \phi_1\phi_2 \rangle, \ldots, L_{p+1}=\langle \phi_1\phi_2^{p-1} \}.$$Observe that if $i \neq j$ then $L_iL_j=L_jL_i$ and $\langle L_i, L_j \rangle=N$, showing that the genus of $S/\langle L_i, L_j\rangle$ is zero. Let $\Pi: S \to S/N$ and $\Pi_i:S \to S/L_i$ be the regular covering maps with deck groups $N$ and $L_i$ respectively, for each $i \in \{1, \ldots, p+1\}.$ Observe that if $q \in S/N$ is a ramification value of $\Pi$ then $\Pi^{-1}(q)$ consists of $p$ points, all of them with $N$-stabilizer $L_j$ for some $j.$   This shows that if  $c_i$ is the number of points of $S$ that are fixed by $L_i$ then 
 $c_1+\cdots+c_{p+1}=(n+1)p.$ The Riemann-Hurwitz formula applied to $\Pi_i$ implies that $2g-2=2p\gamma_i-2p+c_i(p-1),$ where $\gamma_i$ is the genus of $S/L_i.$ Consequently, one has that\begin{equation}\label{hola}(2g-2)(p+1)=2p(\gamma_1+\cdots+\gamma_{p+1})-2p(p+1)+(n+1)p(p-1).\end{equation}

Now, by considering \eqref{gen} with $m=2$, we see that $g=((n-1)p-2)(p-1)/2$ and the equality \eqref{hola} turns into $\gamma_1+\cdots+\gamma_{p+1}=g.$ The proof follows from  \cite[Theorem C]{KR}.
\end{proof}

\section{Example 1: The case $m=2, k=p$ prime and $n=3$}
 For the sake of clarity,  we now specialize our results to the case  $n=3$, that is, $\mathbb{Z}_p^2$-actions of signature $(0; p^{4}),$ and describe some examples in detail. Note that the Riemann surfaces corresponding to such $\mathbb{Z}_p^2$-actions have genus $(p-1)^2$ and form  complex one-dimensional families. In this case, we assume $p \geqslant 3$.

By applying Theorem \ref{max}, one sees that the collection ${\mathcal F}(p,3)$ consists of the following groups.
\begin{enumerate}

\item $K(r,s)=\langle a_{1}^{r}a_{2}^{s}a_{3}^{-1}\rangle,$ where $r,s \in \{0,1,\ldots,p-1\}$ satisfy $(r,s) \notin \{(0,0),(p-1,p-1)\}$.

\item  $K(l)=\langle a_{1}^{l}a_{2}^{-1}\rangle$ where $l \in \{1,\ldots,p-1\}.$
%
%
\end{enumerate}

\begin{example} \label{ej1}
Consider the ${\mathbb Z}_{p}^{2}$-action of signature $(0; p^{4})$ associated to the group $K=K(0, p-1)=\langle a_2^{p-1}a_3^{-1}\rangle.$ By Proposition \ref{pp5}, the family formed by the Riemann surfaces $S_{K}$ are algebraically represented by 
\begin{equation}\label{pepsi}S_{K}: \left\{ \begin{array}{l}
y_{1}^{p}=x(x-1)^{p-1}\\
y_{2}^{p}=(x-\lambda)^{p-1}\\
\end{array}
\right.
\end{equation}where $\lambda \in \mathbb{C}-\{0,1\}.$ A routine computation shows that the maps $$a(x,y_1, y_2)=(\tfrac{\lambda}{x}, \tfrac{\lambda^{1/p}y_2}{x},\tfrac{\lambda^{1-1/p}y_1}{x}) \mbox{ and } b(x, y_1, y_2)=(\tfrac{x-\lambda}{x-1}, \tfrac{(1-\lambda)^{1-1/p}(x-\lambda)}{(x-1)y_2}, \tfrac{(1-\lambda)^{1-1/p}x}{y_1})$$are automorphisms of $S_K$ and $\langle a, b \rangle \cong \mathbb{Z}_2^2$. In addition, one has that $a\phi_1a=\phi_2, \, a\phi_2a=\phi_1, \, b\phi_1b=\phi_2^{-1}, \, b\phi_2b=\phi_1^{-1},$ where $\phi_1$ and $\phi_2$ generate $N \cong \mathbb{Z}_p^2$ and are given in Proposition \ref{pp5}.  Now, if we set $$R:=\phi_1\phi_2, \, S:=b, \, \hat{R}:=\phi_1\phi_2^{-1}, \, \hat{S}:=a \mbox{ then }\mbox{Aut}(S_K) \geqslant \langle S, R \rangle \times \langle \hat{S}, \hat{R}\rangle \cong \mathbf{D}_p \times \mathbf{D}_p.$$ The signature of the action of this last group is $(0; 2,2,2,p)$ and therefore, as a maximal signature \cite{Sing72}, up to finitely many exceptions, the automorphism group of $S_K$ is isomorphic to $\mathbf{D}_p \times \mathbf{D}_p.$  Later we shall see that there is only one exceptional member with more than $4p^2$ automorphisms.

The interest in this family comes from the following fact. A well-known result due to Accola \cite{A84} states that if a $p$-gonal Riemann surface has genus $g > (p-1)$ then the $p$-gonal morphism is unique. The family described above was considered by Costa, Izquierdo and Ying in \cite{CIY10}, when they noticed that it has two $p$-gonal morphisms. With our notation, such morphisms are $R(x,y_1,y_2)=(x, \omega_p y_1, \omega_py_2) \mbox{ and } \hat{R}(x,y_1,y_2)=(x, \omega_p y_1, \bar{\omega}_py_2).$ We should point out that the algebraic description of this family given here differs from the one already known for such surfaces; see \cite[Section 5]{CIY10}.
\end{example}

We recall that the action of   $\langle \Phi_{1}, \Phi_{2} \rangle= {\rm Aut}_{g}(H) \cong {\bf S}_{4}$ on ${\mathcal F}(p,3)$ is given as follows (see Proposition \ref{tp}).
$$\Phi_1(a_1, a_2, a_3, a_4)=(a_2, a_1, a_3, a_4) \mbox{ and } \Phi_2(a_1, a_2, a_3, a_4)=(a_2, a_3, (a_1a_2a_3)^{-1}, a_1).$$  For instance, observe that \begin{equation}\label{red}\Phi_1(K(r,s))=K(s,r),  \Phi_2(K(l))=K(0,l) \mbox{ and }\Phi_2(K(0,s))=K(u, u) \end{equation}where $u \in \{1, \ldots, p-1\}$ satisties $u(1+s) \equiv -1 \mbox{ mod }p$.

\begin{example} 
We proceed to describe explicitly the orbits of this action for $p=5.$ Note that $\mathcal{F}(5,3)$ consists of 27 groups. By considering  \eqref{red}, it suffices to restrict our attention to the groups  $$K(r,s) \mbox{ where } (r,s) \in \{0,1,2,3,4\}^2-\{(0,0), (4,4)\} \mbox{ and } r < s.$$After some computations, one can see that there are exactly four ${\rm Aut}_{g}(H)$-orbits,  represented by 
$K(0,1),$ $K(0,2),$ $K(0,4)$ and $K(1,2).$ All the above coupled with Proposition \ref{pp5} and Corollary \ref{felix} can be summarized as follows. There are exactly four topologically pairwise non-equivalent ${\mathbb Z}_{5}^{2}$-actions of signature $(0; 5^{4})$. Equivalently, the complex one-dimensional family formed by the ${\mathbb Z}_{5}^{2}$-actions of signature $(0; 5^{4})$ consists of four irreducible components. The corresponding Riemann surfaces (of genus $16$) are  represented by the following curves:

{\small
\begin{center}
\begin{tabular}{|c| c ||c | c|} 
 \hline
 $K(0,1)$ & $\left\{\begin{array}{l}
y_{1}^{5}=x(x-1)(x-\lambda)^{3}\\
y_{2}^{5}=(x-\lambda)^{4}
\end{array}
\right.$ & $K(0,2)$ & $\left\{\begin{array}{l}
y_{1}^{5}=x(x-1)^{2}(x-\lambda)^{2}\\
y_{2}^{5}=(x-\lambda)^{4}
\end{array}
\right.
$ \\ 
 \hline
 $K(0,4)$ & $\left\{\begin{array}{l}
y_{1}^{5}=x(x-1)^{4}\\
y_{2}^{5}=(x-\lambda)^{4}
\end{array}
\right.$ & $K(1,2)$ & $\left\{\begin{array}{l}
y_{1}^{5}=x(x-1)^{2}(x-\lambda)^{2}\\
y_{2}^{5}=(x-1)(x-\lambda)^{3}
\end{array}
\right.$ \\
 \hline
\end{tabular}
\end{center}where $\lambda \in {\mathbb C} -\{0,1\}$.
}

It is worth mentioning that complete lists of Riemann surfaces of genus 16 with non-trivial automorphisms are available in the literature. For instance, in the database \cite{K} this family  is labeled as O16.271. However, it seems that in \cite{K} the computation of the number of classes of topological actions for our family lies beyond the scope of the algorithm employed.

\end{example}
%
%
%
%

The determination of the number of ${\rm Aut}_{g}(H)$-orbits in $\mathcal{F}(p,n)$ boils down to a routine --but certainly tedious-- computation. With the help of routines implemented in GAP, we were able to compute the number $N$ of orbits in $\mathcal{F}(p,3)$ for some small primes. This is summarized in the following table.
{\small
$$
\begin{array}{|c|c|c|c|c|c|c|c|c|c|c|}\hline
p & 3 & 5 & 7 & 11 & 13 & 17 & 19 & 23 & 29 & 113 \\\hline
N & 2 & 4 & 6 & 10 & 14 & 20 &24 & 32 & 48 & 580 \\\hline
\end{array}
$$
}

%
%
%
%
%
%
%
%
%
%
%

\subsubsection{\bf Extra automorphisms}
To construct $\mathbb{Z}_p^2$-actions of signature $(0; p^{4})$ endowed with extra automorphisms, we consider the following elements in ${\rm Aut}_{g}(H) \cong {\bf S}_{4}$.
$$\Phi_{3}=(3 \, 4), \quad \Phi_{4}=(1 \, 2)(3 \, 4), \quad \Phi_{5}=(2 \, 3 \, 4),\quad \Phi_{6}= (1 \, 4)(2 \, 3), \quad \Phi_{7}=(2 \, 4).$$

As the only non-trivial finite groups of M\"{o}bius transformations that keep setwise invariant a set of four points in $\bar{\mathbb{C}}$ are isomorphic to $\mathbb{Z}_2, \mathbb{Z}_3, \mathbb{Z}_4, \mathbb{Z}_2^2, \mathbf{D}_4$ and $\mathcal{A}_4,$ the only subgoups of ${\rm Aut}_{g}(H)$ induced by conformal automorphisms of $S_K/N_K \cong C_{p}(\Lambda)/H$  are, up to conjugation,  summarized in the following table.
{\small
\begin{center}
\begin{tabular}{|c| c | c| c||c| c | c | c|} 
 \hline
label & generators & group & permutation & label & generators & group & permutation\\ [0.5ex] 
 \hline\hline
 $Q_1^*$ & $ \Phi_3 $ & $\mathbb{Z}_2$ & $(3 \, 4)$  & $Q_5^*$ & $ \Phi_4, \Phi_6$ & $\mathbb{Z}_2^2$ & $(1 \, 2)(3 \, 4), (1 \, 4)(2 \, 3)$  \\  \hline

$Q_2^*$ & $ \Phi_4$ & $\mathbb{Z}_2$ & $(1 \, 2)(3 \, 4)$  & $Q_6^*$ & $ \Phi_1, \Phi_3$ & $\mathbb{Z}_2^2$ & $(1 \, 2), (3 \, 4)$  \\  \hline

$Q_3^*$ & $ \Phi_5$ & $\mathbb{Z}_3$ & $(2 \, 3 \, 4)$  & $Q_7^*$ & $ \Phi_2, \Phi_7$ & $\mathbf{D}_4$ & $(1 \, 2 \, 3 \, 4), (2 \, 4)$  \\  \hline

$Q_4^*$ & $ \Phi_2$ & $\mathbb{Z}_4$ & $(1 \, 2 \, 3 \, 4)$  & $Q_8^*$ & $ \Phi_4, \Phi_5$ & $\mathcal{A}_4$ & $(1 \, 2)(3 \, 4), (2 \, 3 \, 4)$  \\  \hline
\end{tabular}
\end{center}
}

%
%

We recall that $\mathscr{C}_p({Q_{j}}):=\{ K \in \mathcal{F}(p,3) :   K \mbox{ is } Q_{j}^*\mbox{-invariant}\}.$

\begin{prop}\label{qs} \mbox{}
\begin{enumerate}
\item $\mathscr{C}_p(Q_{1})=\{K(l), K(\tfrac{p-1}{2},\tfrac{p-1}{2}) : l=1,\ldots,p-1\}$.

\item $\mathscr{C}_p(Q_{2})=\{K(1), K(p-1), K(r,p-1-r) : r=0,\ldots,p-1\}$.

\item If $p=3$ or $p \equiv 2 \mbox{ mod }3$ then $\mathscr{C}_p(Q_{3}) = \emptyset$, and if $p \equiv 1 \mbox{ mod }3$ then $\mathscr{C}_p(Q_{3})=K(\tfrac{s}{1-s},s): s \in \mathcal{P}_1\},$ where $\mathcal{P}_1=\{s \in \{1, \ldots, p-1\}: s^{2}+s+1 \equiv 0 \mbox{ mod } p \}.$

\item If $p \equiv 3 \mbox{ mod }4$ then $\mathscr{C}_p(Q_{4})=\{K(p-1,0)\},$ and if $p \equiv 1 \mbox{ mod }4$ then $\mathscr{C}_p(Q_{4})=\{K(p-1,0), K(\tfrac{s}{1-s},s): s \in \mathcal{P}_2\},$ where $\mathcal{P}_2=\{ s \in \{1, \ldots, p-1\}: s^2+2s+2 \equiv 0 \mbox{ mod }p\}.$

\item $\mathscr{C}_p(Q_{5})=\{K(p-1), K(0,p-1), K(p-1,0)\}$.

\item $\mathscr{C}_p(Q_{6})=\{K(1), K(p-1), K(\tfrac{p-1}{2},\tfrac{p-1}{2})\}$.

\item $\mathscr{C}_p(Q_{7})=\{K(p-1,0)\}$.

\item $\mathscr{C}_p(Q_{8})=\emptyset$.
\end{enumerate}
\end{prop}
\begin{proof} We shall only check the first and last statement, as the remaining ones are proved analogously. First, observe that $\Phi_{3}(K(l))=K(l)$ for each $l$, and therefore $K(l) \in \mathscr{C}_p(Q_{1})$. Now, notice that $$\Phi_{3}(K(r,s))=\langle a_{1}^{r}a_{2}^{s}a_{4}^{-1}\rangle=\langle a_{1}^{1+r}a_{2}^{1+s}a_{3}\rangle=\langle a_{1}^{-1-r}a_{2}^{-1-s}a_{3}^{-1}\rangle=K(p-1-r,p-1-s).$$
It follows that $\Phi_{3}(K(r,s))=K(r,s)$ if and only if  $r=s=\tfrac{p-1}{2}$, proving (1). Note that $Q_2^* \leqslant Q_8^*$ and therefore 
$$\mathscr{C}_p(Q_{8}) \subset \mathscr{C}_p(Q_{2})=\{K(1), K(p-1), K(r,p-1-r): r=0,\ldots,p-1\}.$$
Note that 
$\Phi_{5}(K(1))$ and $\Phi_{5}(K(p-1))$ are groups of type (2). Moreover, 
$$\Phi_{5}(K(r,p-1-r))=\langle a_{1}^{r}a_{3}^{-1-r}a_{4}^{-1}\rangle=\langle a_{1}^{1+r}a_{2}a_{3}^{-r}\rangle.$$
So, for $\Phi_{5}(K(r,p-1-r))=K(r,p-1-r)$, we must have that $r \neq 0$ satifies $
r^2-r-1 \equiv 0 \mbox{ mod }p \, \mbox{ and } \, r^{2}+r+1 \equiv 0 \mbox{ mod }p,$
which is clearly impossible. This proves (8).
\end{proof}

\begin{example}[Continuation of Example \ref{ej1}]
The fact that $K(0,p-1) \in \mathscr{C}_p(Q_5^*)$ guarantees that the members of the family of Riemann surfaces \begin{equation}\label{pepsi6}S_{K(0, p-1)}(\lambda): \left\{ \begin{array}{l}
y_{1}^{p}=x(x-1)^{p-1}\\
y_{2}^{p}=(x-\lambda)^{p-1}\\
\end{array}
\right.\mbox{ where } \lambda \in \mathbb{C}-\{0,1\},
\end{equation}admit extra automorphisms, and for them $\mbox{Aut}(S_{K(0,p-1)}(\lambda)) \geqslant \tilde{G}\cong \mathbb{Z}_p^2 \rtimes \mathbb{Z}_2^2.$ This situation was already studied in detail in Example \ref{ej1}.

Observe that the group $Q_{7}^{*}$ contains $Q_{5}^{*}$, but $K(0,p-1) \notin \mathscr{C}_p(Q_{7}^*)=\{K(p-1,0)\}$. However, since $K(0,p-1)=\Phi_{1}(K(0,p-1))$ and $\Phi_{1}$ normalizes $Q_{5}^{*}$, we note that  
$$\tilde{Q}_7^*  =\Phi_1Q_7^*\Phi_1  = \langle\Phi_1 \Phi_2 \Phi_1=(1\,3\,4\,2), \Phi_1\Phi_7\Phi_1=(1 \, 4)\rangle$$ contains $Q_{5}^{*}$ and $\mathscr{C}_p(\tilde{Q}_{7}^*)=\{K(0,p-1)\}$. In particular, there are some members of the family \eqref{pepsi6} admitting more than $4p^2$ automorphisms. This observation, coupled with the fact that $\tilde{Q}_7^*$ is not contained in any group $Q_i^*$ (nor in any conjugate), shows that if a member of  \eqref{pepsi6}  has more than  $4p^2$ automorphisms, then $$\mbox{Aut}(S_{K(0,p-1)}(\lambda)) \cong \mathbb{Z}_p^2 \rtimes \mathbf{D}_4.$$


A computation shows that $S_{K(0,p-1)}(\lambda)$ has an automorphism of order 4 if and only if $\lambda=\tfrac{1}{2}$, and this automorphism is given by $$c(x, y_1, y_2)=(\tfrac{2x-1}{2x}, \tfrac{2x-1}{2^{2-1/p}xy_2}, \tfrac{y_1}{2^{1-1/p}}x).$$The automorphism group of this special member acts with signature $(0; 2,4,2p).$  
\end{example}

\begin{rema}
The fact that $K(p-1)$ and $K(p-1,0)$ are $\mbox{Aut}_g(H)$-equivalent to $K(0, p-1)$ tells to us that the discussion above can also be carried out with $K(p-1)$ and $K(p-1,0)$. In these cases, we would obtain different --but equivalent-- algebraic descriptions of the same family and of their automorphisms.
\end{rema}

%
%
%
%
%
%
%
%
%
%


\section{Example 2: The case $m=2, k=p$ prime and $n=5$ }
Let $p \geqslant 2$ be a prime number. The pencil $\mathscr{C}_p$ formed by the smooth complex projective algebraic curves $C_t$ of genus $(p-1)(2p-1)$  given by $$x^{2p}+y^{2p}+z^{2p}+t(x^py^p+y^pz^p+x^pz^p)=0 \, \mbox{ where } t \in \bar{\mathbb{C}}-\{-1, \pm 2\}$$has been recently studied in \cite{MRC}. This pencil is a generalization of the pencil of Kuribayashi-Komiya quartics $\mathscr{C}_2$, also known in the literature as the KFT family (see, for instance, \cite{KK77} and \cite{KFT}). Note that $$N=\langle [x:y:z] \mapsto [\omega_p x : y: z], [x:y:z] \mapsto [x:\omega_p y : z] \rangle\cong \mathbb{Z}_p^2$$ is a group of automorphisms of each member $C_t$ of $\mathscr{C}_p.$ Moreover, the quotient $C_t/N$ has genus zero, and hence $(C_t, N)$ is a $\mathbb{Z}_p^2$-action of signature $(0; p^6).$ Furthermore, as proved in \cite[Proposition 1]{MRC}, each $C_t$ is endowed with a group of automorphisms $\mathbf{G}_p$ isomorphic to the semidirect product $\mathbb{Z}_p^2 \rtimes \mathbf{D}_3$ given by $$ \mathbf{G}_p \cong \langle a,b,r,s : a^p=b^p=[a,b]=r^3=s^2=(sr)^2=1,rar^{-1}=(ab)^{-1}, rbr^{-1}=a, sas=(ab)^{-1}, [s,b]=1  \rangle.$$The group $\mathbf{G}_p$ acts on $C_t$ with signature $(0; 2,2,3,p).$ Note that $N \trianglelefteq \mathbf{G}_p$ and $\mathbf{G}_p/N \cong \mathbf{D}_3.$

Motivated by the above, we apply our results to study $\mathbb{Z}_p^2$-actions of signature $(0; p^6)$ that admit extra automorphisms and that form complex one-dimensional families.

Let $(S, N, G)$ be triple such that $(S, N)$ is a $\mathbb{Z}_p^2$-action of signature $(0; p^6)$ and $S$ is endowed with a group of automorphisms $G$ such that $N \trianglelefteq G \leqslant \mbox{Aut}(S).$  Observe that the genus of $S$ is $(p-1)(2p-1).$ The quotient $G/N$ is a isomorphic to a non-trivial finite group $L$ of M\"{o}bius transformations that keeps setwise invariant a set of six points in $\bar{\mathbb{C}}$. If we require that the corresponding quotient has moduli one (namely, the orbifold $(S/N)/L$ has  four cone points), then there are two scenarios. 
\begin{enumerate}
\item $L$ is isomorphic to $\mathbf{D}_3$ and the signature of the action of $G$ is $(0; 2,2,3,p),$ and
\item $L$ is isomorphic to $\mathbb{Z}_2^2$ and and the signature of the action of $G$ is  $(0; 2,2,p,2p)$
\end{enumerate}
We proceed to study these cases separately.

\subsection{The case $L\cong \mathbf{D}_3$}
Note that if $p \neq 3$ then $G$ is isomorphic to a semidirect product of the form $\mathbb{Z}_p^2 \rtimes \mathbf{D}_3.$ Set  
$$L \cong G/N \cong \mathbf{D}_3= \langle a, b: a^3=b^2=(ab)^2=1\rangle.$$

Assume that the regular covering map $\pi: S \to S/N$ ramifies over $\infty, 0, 1, q_4, q_5, q_6$. It follows that, if we write $\Lambda =(q_4, q_5, q_6) \in \Omega_5,$ then 
$S \cong C_{p}(\Lambda)/K_S$ and $N \cong H/K_S$ for some $K_S \in \mathcal{F}(p,5,2).$
After conjugating by a suitable M\"{o}bius transformation, we can assume 
$a(z)=\tfrac{1}{1-z}$, $b(z)=\tfrac{\lambda z-\lambda+1}{z-\lambda}$  and that  $q_4=\lambda$, $q_5=\tfrac{1}{1-\lambda}$, $q_6=\tfrac{\lambda-1}{\lambda}.$
Note that $(q_4, q_5, q_6) \in \Omega_5$ if and only if $\lambda \neq \tfrac{1}{2}(-1 \pm i \sqrt{3}).$ The group $L$ lifts to a group of automorphisms $Q_{S,G}$ of  
\begin{equation*}
C_p(\Lambda) :  \left \{ \begin{array}{ccccccc}
x_1^p & + & x_2^p & + & x_3^p & =  & 0\\
\lambda x_1^p & + & x_2^p & + & x_4^p & =  & 0\\
\tfrac{1}{1-\lambda} x_1^p & + & x_2^p & + & x_5^p & =  & 0\\
\tfrac{\lambda-1}{\lambda} x_1^p & + & x_2^p & + & x_6^p & =  & 0\\
\end{array} \right\} \subset \mathbb{P}^5\end{equation*} such that $Q_{S,G}/K_S \cong G$ and $S/G \cong C_{p}(\Lambda)/{Q_{S,G}}.$

\begin{lemm} \label{ttmm} The group $Q_{S,G}$ is isomorphic to a semidirect product $$\mathbb{Z}_p^5 \rtimes  \mathbf{D}_3 =\langle a_1, \ldots, a_5 : a_j^p=[a_i, a_j]=1 \rangle \rtimes \langle A, B : A^3=B^2=(AB)^2=1\rangle,$$where the action by conjugation of $A$ and $B$ on $a_1, \ldots, a_5$ is given by$$A: (a_1, a_2, a_3, a_4, a_5) \mapsto (a_2, a_3, a_1, a_5, (a_1\cdots \, a_5)^{-1}) \, \mbox{ and } \, B:  (a_1, a_2, a_3, a_4, a_5)  \mapsto (a_4, (a_1\cdots \, a_5)^{-1}, a_5, a_1, a_3).$$
\end{lemm}
\begin{proof}
As proved in \cite[Corollary 9]{GHL}, the fact that $a$ induces the permutation $(1 \, 2 \, 3)(4 \, 5 \, 6)$ on the subindices of $q_1, \ldots, q_6$ implies that every lifting $A$ of $a$ 
is of the form $$[x_1: \ldots : x_6] \mapsto [x_3 : A_2x_1:A_3x_2: A_4x_6: A_5x_4:A_6x_5],$$
where $A_2, \ldots, A_6$ are arbitrary complex numbers that satisfy $A_2^p=A_3^p=1, A_4^{p}=-\lambda, A_5^p=\tfrac{1}{\lambda-1} \mbox{ and } A_6^{p}=\tfrac{1-\lambda}{\lambda}.$

Similarly, every lifting $B$ of $b$ is of the form
$$[x_1: \ldots : x_6] \mapsto [x_4 : B_2 x_6: B_3 x_5: B_4x_1: B_5x_3: B_6x_2],$$
where 
$B_{2}^{p}=-\lambda, B_{3}^{p}=\lambda-1, B_{4}^{p}=\lambda^{2}-\lambda+1, B_{5}^{p}=\tfrac{\lambda^{2}-\lambda+1}{\lambda-1}, B_{6}^{p}=-\tfrac{\lambda^{2}-\lambda+1}{\lambda}.$ 
Observe that 
$$A^3([x_1: \ldots : x_6])=[A_2A_3x_1 : A_2A_3x_2:A_2A_3x_3: A_4A_5A_6x_4: A_4A_5A_6x_5:A_4A_5A_6x_6],$$
$$B^2([x_1: \ldots : x_6])=[B_4x_1 : B_2B_6x_2:B_3B_5x_3: B_4x_4: B_5B_3x_5:B_6B_2x_6],$$whereas 
$(AB)^{2}([x_1: \ldots : x_6])$ equals to $$[A_5B_3B_4x_1 : A_2A_4B_6x_2: A_3A_6B_2B_5 x_3: A_2A_4B_6x_4: A_5B_3B_4x_5: A_3A_6B_2B_5x_6].$$

Now, once fixed values for $\lambda^{1/p}$, $(\lambda-1)^{1/p}$ and $(\lambda^{2}-\lambda+1)^{1/p}$ are chosen, if we take
$$A_{2}=A_{3}=1, A_{4}=B_{2}=-\lambda^{1/p}, A_{5}=B_{3}^{-1}=\tfrac{1}{(\lambda-1)^{1/p}}, A_{6}=-\tfrac{(\lambda-1)^{1/p}}{\lambda^{1/p}},$$
$$
B_{4}=(\lambda^{2}-\lambda+1)^{1/p}, B_{5}=\tfrac{(\lambda^{2}-\lambda+1)^{1/p}}{(\lambda-1)^{1/p}},  B_{6}=-\tfrac{(\lambda^{2}-\lambda+1)^{1/p}}{\lambda^{1/p}},
$$
then $A^{3}=B^{2}=(AB)^{2}=1$. Note that $$Aa_1A^{-1}([x_1: \ldots : x_6])=A([\omega_p\tfrac{x_2}{A_2}: \tfrac{x_3}{A_3}: x_1 : \tfrac{x_4}{A_4}: \tfrac{x_5}{A_5} : \tfrac{x_6}{A_6}])=[x_1: \omega_px_2: x_3 : \ldots : x_6]=a_2.$$The remaining relations are obtained analogously. 
\end{proof}

\begin{rema} \label{FH} There are some exceptional values of $\lambda$.

{\bf{(a)}} There is the possibility for the existence of a M\"obius transformation $c$ satisfying that $c^{2}=a$ such that
$$c: q_{4} \mapsto q_{1} \mapsto q_{5} \mapsto q_{2} \mapsto q_{6} \mapsto q_{3} \mapsto q_{4}.$$

This asserts that $\lambda=\lambda_{0} \in \{2,\omega_{6}, \bar{\omega}_{6}\}$ and $c(z)=\frac{z+(1-\lambda_{0})^{2}}{(1-\lambda_{0})(z-\lambda_{0})}.$ For instance, if $\lambda_0=2$ then $c(z)=\tfrac{z+1}{2-z}$ permutes cyclically $\infty, q_5=-1, 0, q_6=\tfrac{1}{2}, 1, q_4=2$, and $c^2=a.$ Set $\Lambda_0=(2, -1, \tfrac{1}{2}).$ By arguing as in the previous proposition, $c$ lifts to an automorphisms $C$ of $C_p(\Lambda_0)$ which satisfies $C^6=[C, A]=(CB)^2=1$. Also, the action by conjugation of $C$ on $a_1, \ldots, a_5$ is given by 
$C: (a_1, a_2, a_3, a_4, a_5) \mapsto (a_5, (a_1 \cdots \, a_5)^{-1}, a_4, a_1,a_2).$ Note that $\langle C, B\rangle \cong \mathbf{D}_6$ and $\langle Q_{S, G}, C \rangle \cong \mathbb{Z}_p^5 \rtimes \mathbf{D}_6.$

\s

{\bf{(b)}} Similarly, if $\lambda=\lambda_1=\pm i$ and $\Lambda_1=(\pm i, \tfrac{1 \pm i}{2}, 1 \pm i)$ 
then the M\"{o}bius transformation $d(z)= \pm i +1$  lifts to an automorphisms $D$ of order 4 of $C_p(\Lambda_1)$ such that $\langle D, A\rangle \cong \mathbf{S}_4$ and $\langle Q_{S,G}, D\rangle \cong \mathbb{Z}_p^5 \rtimes \mathbf{S}_4$. The action by conjugation of $D$ on $a_1, \ldots, a_5$ is given by $D: (a_1, a_2, a_3, a_4, a_5) \mapsto (a_1, a_3, (a_1 \cdots \, a_5)^{-1}, a_2,a_5).$

\end{rema}

We recall that, as introduced in \S\ref{state2}, 
${\mathcal N}_{Q_{S,G}} =\{\Phi \in {\rm Aut}_{g}(H) :  \Phi Q^{*}_{S,G} \Phi^{-1}=Q^{*}_{S,G}\} \leqslant {\rm Aut}_{g}(H)$ is the normalizer of $Q^{*}_{S,G}$ in ${\rm Aut}_{g}(H)$. Abusing notation, we shall denote by $A \in Q^*_{S, G}$ the image of $A$ by $\rho_{Q_{S,G}}$ (see \eqref{repre222}).  Consider the natural group isomorphism $\sigma : \mbox{Aut}_g(H) \to \mathbf{S}_6=\mbox{Sym}\{a_1, \ldots, a_6\}$, and write $u=\sigma(A) =(1 \, 2 \, 3)(4 \, 5 \, 6) \mbox{ and } v:=\sigma(B)=(1 \, 4)(2 \, 6)(3 \, 5).$ Then ${\mathcal N}_{Q_{S,G}}$ is isomorphic to the normalizer $N$ of $\langle u, v \rangle$ in $\mathbf{S}_6,$ which is isomorphic to $\mathbf{D}_3 \times \mathbf{D}_3$ and 
$N/\langle u, v \rangle =\{ \mbox{id, } (4 \, 5 \, 6), (4 \, 6 \, 5), (2 \, 3)(5 \, 6), (2\, 3)(4\, 5), (2\, 3)(4 \, 6)\} \cong \mathbf{D}_3$. All the above is the proof of the following proposition.

\begin{prop}\label{nor}
${\mathcal N}_{Q_{S,G}}$ is generated by $Q_{S,G}^*$ and the geometric automorphisms of $H$ given by 
$$\Psi_2(a_1, \ldots, a_6)=(a_1, a_2, a_3, a_5, a_6, a_4) \, \mbox{ and } \, \Psi_5(a_1, \ldots, a_6)=(a_1, a_3, a_2, a_6, a_5, a_4).$$
\end{prop}

We recall that, following Corollary \ref{coca7}, the cardinality of $\mathscr{C}_p(Q_{S,G})/{\mathcal N}_{Q_{S,G}}$ agrees with the number of pairwise topologically inequivalent triples $(\hat{S},\hat{N},\hat{G})$, where $(\hat{S},\hat{N})$ is a ${\mathbb Z}_{p}^{2}$-action of signature $(0; p^6)$, and $\hat{N} \triangleleft \hat{G} \leqslant {\rm Aut}(\hat{S})$ is such that $\hat{G}/\hat{N}$ induces $Q^{*}_{S,G}$. Here $\mathscr{C}_p({Q_{S,G}})=\{ K \in \mathcal{F}(p,5,2) :   K \mbox{ is } Q_{S,G}^*\mbox{-invariant}\}.$

\begin{prop} \label{verde}
If $K \in \mathscr{C}_p({Q_{S,G}})$ then one of the following statements holds.
\begin{enumerate}
\item $p=3$ or $p \equiv 1 \mbox{ mod }3$, there is $l \in \{1, \ldots, p-1\}$ satisfying $l^2+l+1 \equiv 0 \mbox{ mod }p$ and $$K=K(l):=\langle a_1a_2a_3, a_1^la_2^{-1}, a_4^{l}a_6^{-1}\rangle.$$

\item There are $r, s \in \{0, \ldots, p-1\}$ satisfying $r^2+s^2-rs \equiv 1 \mbox{ mod }p$ and $$K=K(r, s):=\langle a_1a_2a_3, a_1^r a_2^s a_4^{-1}, a_1^{-s}a_2^{r-s}a_5^{-1} \rangle.$$
\end{enumerate}
\end{prop}

\begin{rema}
The group $K(r,s)$, in part (2) of the above proposition, corresponds to the group in case (1) of Theorem \ref{max} with the following parameters:
$$r_{3}=s_{3}=p-1, \,\,\, r_{4}=r, \,\,\, s_{4}=s, \,\,\, r_5 = \left\{ \begin{array}{cc}
p-s & s>0\\
0 & s=0\\
\end{array}
\right. \,\, \mbox{ and } \,\, s_5 = \left\{ \begin{array}{cc}
r-s &  r \geqslant s\\
p+r-s & r<s\\
\end{array}
\right.
$$
\end{rema}

\begin{proof}[\bf Proof of Proposition \ref{verde}]
Let $K \in \mathscr{C}_p({Q_{S,G}})$ and let $\theta: H \to \mathbb{Z}_p^2$ be a  group epimorphism such that $K=\mbox{ker}(\theta).$ As $a_1 \notin K,$ we have that $\phi_1:=\theta(a_1)$ has order $p.$

\noindent
{\bf 1.} Assume $\theta(a_2) \in \langle \phi_1 \rangle.$ The fact that $A(a_1)=a_2, A(a_2)=a_3,  A(a_3)=a_1$ implies that $$\theta(a_2)=\phi_1^l \mbox{ and }\theta(a_3) =\phi_1^{l^2} \mbox{ for some } l \in \{1, \ldots, p-1\} \mbox{ such that }1+l+l^2 \equiv 1 \mbox{ mod }p.$$If we write $\phi_2:=\theta(a_4)$ then the we have that $a_5=B(a_3)$ and therefore $\theta(a_5)=\phi_2^{l^2}.$ Similarly, $\theta(a_6)=\phi_2^{l}.$ Note that $\phi_2 \notin \langle \phi_1 \rangle$ and $\langle \phi_1, \phi_2 \rangle \cong \mathbb{Z}_p^2$. It follows that  $K=\langle a_1a_2a_3, a_4a_5a_6, a_1^la_2^{-1}, a_1^{l^2}a_3^{-1}, a_2^{l^2}a_5^{-1}, a_2^l a_6^{-1}\rangle=\langle a_1a_2a_3,  a_1^la_2^{-1}, a_2^l a_6^{-1}\rangle=K(l).$

\noindent
{\bf 2.} Assume $\theta(a_2) \notin \langle \phi_1 \rangle.$ If we write $\phi_2:=\theta(a_2)$ then $\langle \phi_1, \phi_2 \rangle \cong \mathbb{Z}_p^2$. The action of $A$ implies that, if we write \begin{equation}\label{para0}\theta(a_3)=\phi_1^{\hat{r}}\phi_2^{\hat{s}}, \, \mbox{ then } \, \hat{r}\hat{s} \equiv 1 \mbox{ mod }p \mbox{ and }\hat{r}+\hat{s}^2 \equiv 0 \mbox{ mod }p.\end{equation}Similarly, if we set $\phi(a_i)=\phi_1^{r_i}\phi_2^{s_i} \mbox{ for }i=4,5,6$ then the action of $A$ show that \begin{equation}\label{para}r_5 \equiv \hat{r}s_4 \mbox{ mod }p, \, s_5 \equiv r_4 + \hat{s}s_4\mbox{ mod } p,  \, r_6 \equiv \hat{r}r_4 + s_4\mbox{ mod }p, \, s_6 \equiv \hat{s}r_4 \mbox{ mod }p.\end{equation}Now, the action of $B$ on $a_4$ implies that $(\phi_1^{r_4}\phi_2^{s_4})^{r_4}(\phi_1^{r_6}\phi_2^{s_6})^{s_4}=\phi_1,$ showing that $$s_4(r_4+s_6) \equiv 0 \mbox{ mod }p \, \mbox{ and } r_4^2+r_6s_4 \equiv 1 \mbox{ mod  } p.$$

We consider each possible case separately.

\noindent
{\bf 2.1} Assume $s_4=0,$ and therefore  $r_4= \pm 1.$ By \eqref{para} we have that $r_5=0, s_5=r_4, r_6=\hat{r}r_4$ and $s_6=\hat{s}r_4$ and therefore $\theta$ is given by 
$(a_1, \ldots, a_6) \mapsto (\phi_1, \phi_2, \phi_1^{\hat{r}}\phi_2^{\hat{s}}, \phi_1, \phi_2, \phi_1^{\hat{r}}\phi_2^{\hat{s}}) \mbox{ or }(\phi_1, \phi_2, \phi_1^{\hat{r}}\phi_2^{\hat{s}}, \phi_1^{-1}, \phi_2^{-1}, \phi_1^{-\hat{r}}\phi_2^{-\hat{s}})$ according to whether or not $r_4=1.$ By considering that $a_1 \cdots \, a_6 =1$ in the former case, and the action of $B$ coupled with \eqref{para0} in the latter, we obtain  $\hat{r}=\hat{s}=-1.$ We  conclude that $K=K_3:=\langle a_1a_2a_3, a_1a_4^{-1}, a_2a_5^{-1} \rangle \, \mbox{ or } \, K=K_4:=\langle a_1a_2a_3,  a_1a_4, a_2a_5\rangle,$ according to whether or not $r_4=1.$

\noindent
{\bf 2.2} Assume $r_4=-s_6$, and therefore $r_4^2+r_6s_4=1.$ By \eqref{para} we have that $s_6=\hat{s}r_4,$ and therefore $r_4=s_6=0$ or $\hat{s}=-1$. 

\noindent
{\bf 2.2.1} Assume $r_4=s_6=0$. We have $r_5=\hat{r}s_4, s_5=\hat{s}s_4, r_6=s_4$, and $s_4= \pm 1.$ It follows that $\theta$ is given by 
$$(a_1, \ldots, a_6) \mapsto (\phi_1, \phi_2, \phi_1^{\hat{r}}\phi_2^{\hat{s}}, \phi_2^{\pm 1}, \phi_1^{\pm \hat{r}}\phi_2^{\pm \hat{s}}, \phi_1^{\pm 1}).$$The action of $B$ implies that $\hat{r}=\hat{s}=-1$ and we then obtain that $\theta$ is given by $$(\phi_1, \phi_2, \phi_1^{-1}\phi_2^{-1}, \phi_2, \phi_1^{-1}\phi_2^{-1}, \phi_1) \, \mbox{ or }(\phi_1, \phi_2, \phi_1^{-1}\phi_2^{-1}, \phi_2^{-1}, \phi_1\phi_2, \phi_1^{-1}).$$Thus, we  conclude that $K=K_1:=\langle a_1a_2a_3,  a_1a_6^{-1}, a_2a_4^{-1}\rangle$ or $K=K_2:=\langle a_1a_2a_3, a_1a_6, a_2a_4 \rangle$, respectively.

\noindent
{\bf 2.2.2} Assume $\hat{s}=-1$. We have that $\hat{r}=-1$ and therefore $r_5=-s_4, s_5=r_4-s_4, r_6=s_4-r_4$ and $s_6=-r_4.$ By proceeding as before, the action of $B$ shows that $r_4^2+s_4^2-r_4s_4 \equiv 1 \mbox{ mod }p$. Thus, if we let $r:=r_4$ and $s:=s_4$ then $\theta$ is given by $(a_1, \ldots, a_6) \mapsto (\phi_1, \phi_2, \phi_1^{-1}\phi_2^{-1},  \phi_1^{r}\phi_2^{s},  \phi_1^{-s}\phi_2^{r-s},  \phi_1^{s-r}\phi_2^{-r})$, and therefore $K=K(r,s).$

The fact that $K(0,1)=K_1, K(0, p-1)=K_2, K(1,0)=K_3$ and $K(p-1,0)=K_4$ finishes the proof. 
\end{proof}

\begin{theo} Let $p \geqslant 2$ be a prime number. Consider the set $$F_p=\{(r,s) : 2 \leqslant s < r \leqslant p-2 \mbox{ and } r^2+s^2-rs \equiv 1 \mbox{ mod }p \}$$and let $\gamma$ be its cardinality. Write $$\alpha := \left\{ \begin{array}{cc}
0 & \text{if } p \equiv 2 \mbox{ mod }3\\
1 & \text{ otherwise }\\
\end{array}
\right. \,\,\, \mbox{ and }\,\,\, \beta := \left\{ \begin{array}{cc}
 1 & \text{if } p=2\\
2 & \text{otherwise } \\
\end{array}
\right.$$Then the cardinality of $\mathscr{C}_p(Q_{S,G})/{\mathcal N}_{Q_{S,G}}$ is $\alpha+\beta+\tfrac{1}{3}\gamma.$ 
\end{theo}

\begin{proof}
Following Proposition \ref{nor}, the group ${\mathcal N}_{Q_{S,G}}$ is generated by $Q_{S,G}^*$ and by $\Psi_2,\Psi_5.$ The fact that each $K \in \mathcal{F}(p,5,2)$ is $\langle A, B \rangle$-invariant shows that we only need to determine the number of orbits of the action of $\langle \Psi_2,\Psi_5 \rangle \cong \mathbf{D}_3$ on $\mathscr{C}_p(Q_{S,G}).$ For simplicity, we write $\Psi_2=(4 \, 5 \, 6)$, $\Psi_2^2= (4 \, 6 \, 5)$, $\Psi_3=\Psi_5\Psi_2^2= (2 \, 3)(5 \, 6)$, and $\Psi_5=(2\, 3)(4 \, 6).$
Observe that the groups $K(l)$ appear if and only if $p=3$ or $p \equiv 1 \mbox{ mod }3,$ and in such a case there are exactly two groups of this type; namely $K(l)$ and $K(l^2)$. Note that 
$\Psi_3(K(l))=\langle a_1a_2a_3, a_1^la_3^{-1}, a_4^la_5^{-1}\rangle.$
The equalities $(a_1a_2a_3 \cdot a_1^la_3^{-1})^{-1}=a_1^{l^2}a_2^{-1} \mbox{ and }((a_1a_2a_3)\cdot(a_4^la_5^{-1}))^{-1}=a_4^{l^2}a_6^{-1}$ ensures that  $\Psi_3(K(l))=K(l^2).$ In a very similar way, it can be seen that $\Psi_2(K(l))=K(l)$ and $\Psi_2(K(l^2))=K(l^2).$
It follows that $K(l)$ and $K(l^2)$ form $\alpha$ orbits. We now consider the groups $K(r,s)$ where $r=0$ or $s=0$ or $r=s.$ Observe that these cases yield only six groups, namely, $K(0,1)$, $K(0, p-1)$, $K(1,0)$, $K(p-1,0)$, $K(1,1)$ and $K(p-1,p-1).$
We claim that they form $\beta$ orbits. In fact, a routine computation shows that $\{K(0,1), K(1,0), K(p-1,p-1)\}$ and $\{K(0,p-1), K(p-1,0), K(1,1)\}$ are two orbits if $p \geqslant 3$, and that they colapse in one orbit when $p=2.$ All the above coupled with the equality $\Psi_5(K(r,s))=K(s,r)$ allows us to restrict our attention to the groups $K(r,s)$ where $(r,s)$ belongs to the parameter set $F_p.$ We write $\mathcal{F}_p=\{K(r,s): (r,s) \in F_p\}.$ Note that $F_2$ and $F_3$ are empty. Therefore we assume $p \geqslant 5.$ If $(r,s) \in F_p$ then $\Psi_5(K(r,s))=K(s,r) \notin \mathcal{F}_p.$ Thus, the set $\mathcal{F}_p$ splits into orbits of length 1 or 3. The equalities 
$\Psi_2(K(r,s))=K(p+s-r, p-r)$ and $\Psi_2^2(K(r,s))=K(p-s, r-s)$, together with the fact that $(r,s), (p-s, r-s)$ and $(p+s-r, p-r)$ are pairwise distinct, show that  $\mathcal{F}_p$ splits into orbits of length 3 only. This completes the proof. 
\end{proof}

\begin{example}
If $p=2$ then $\mathscr{C}_2({Q_{S,G}})$ consists of only three groups: $K(0,1), K(1,1)$ and $K(1,0)$, and they are equivalent. We then obtain the well-known fact that among the topological classes of actions of $\mathbb{Z}_2^2$ on Riemann surfaces $S$ genus $g=3$ with signature $(0; 2^6)$, there is only one of them for which $S$ has a group of automorphisms $G \cong Q_{S,G}/K(0,1)$ which is isomorphic to  $$\langle a_1, a_2, A, B : a_1^2=a_2^2=[a_1, a_2]=A^3=B^2=(AB)^2=1,Aa_1A=a_2, Aa_2A=a_1a_2, Ba_1B=a_2 \rangle \cong \mathbf{G}_2 \cong \mathbf{S}_4,$$ and acts with signature $(0; 2,2,2,3).$ These Riemann surfaces form the family of quartics $$x^4+y^4+z^4+t(x^2y^2+x^2z^2+y^2z^2)=0,$$ denoted by $\mathscr{C}_2$ at the beginning of this section. We apply Proposition \ref{pp5} to obtain that 
\begin{equation*}
\left\{ \begin{array}{l}
y_{1}^{2}=x(x-1)(x-\lambda)(x-\tfrac{1}{1-\lambda})\\
y_{2}^{2}=(x-1)(x-\tfrac{1}{1-\lambda}) (x-\tfrac{\lambda-1}{\lambda})\\
\end{array}
\right.\mbox{ where } \lambda \in \mathbb{C}-\{0,1\},
\end{equation*}is another algebraic description for this family. Observe that, while the (generic) members of this family are non-hyperelliptic, they are the fiber product of two Riemann surfaces of genus one.

Furthermore, with the notations of Remark \ref{FH}, observe that $C(K(0,1))=\langle a_4a_5a_6, a_1a_6, a_2a_5a_6\rangle=K(0,1)$ and therefore $K(0,1)$ is $C$-invariant. Similarly,  $K(1,1)$ is $D$-invariant. We then conclude that $S_0:= C_2(\Lambda_0)/K(0,1)$ and $S_1:=C_2(\Lambda_1)/K(1,1)$ are members of $\mathscr{C}_2$ admitting a group of automorphisms isomorphic to $\mathbb{Z}_2^2 \rtimes \mathbf{D}_6 \cong \mathbb{Z}_2 \times \mathbf{S}_4$ and $\mathbb{Z}_2^2 \rtimes \mathbf{S}_4 \cong \mathbb{Z}_4^2 \rtimes \mathbf{S}_3$ respectively. These Riemann surfaces are the unique hyperelliptic curve of genus 3 with 48 automorphisms, and the Fermat quartic, respectively. 
\end{example}

\begin{example}
If $p=3$ then $\mathscr{C}_3({Q_{S,G}})$ consists of seven groups, and they split into three classes, represented by $K(0,1), K(0,2)$ and $K(1)$. It follows that there are precisely three classes of topologically inequivalent triples $(S, N, G)$, where $(S, N)$ is a $\mathbb{Z}_3^2$-action of signature $(0; 3^6)$ such that $N \trianglelefteq G$ and $G/N \cong  \mathbf{D}_3$. In this case, $G$ acts with signature $(0; 2,2,3,3).$ Observe that there are at most three groups $G$ as before. In the former case, we have   $$G = G(0,1):= Q_{S, G}/K(0,1)=\langle a_1, \ldots, a_6, A, B\rangle/\langle a_1=a_6, a_2=a_4, a_3=a_5=(a_1a_2)^{-1}\rangle$$
$$ \cong \langle a_1, a_2, A, B : \ldots, Aa_1A^{-1}=a_2, Aa_2A^{-1}=(a_1a_2)^{-1}, Ba_1B=a_2 \rangle.$$Similarly, in the second and third case, we obtain that $$G = G(0,2):= Q_{S, G}/K(0,2) \cong \langle a_1, a_2, A, B : \ldots, Aa_1A^{-1}=a_2, Aa_2A^{-1}=(a_1a_2)^{-1}, Ba_1B=a_2^{-1} \rangle$$ $$G = G(1):= Q_{S, G}/K(1) \cong \langle a_1, a_2, A, B : \ldots, [A,a_1]=[A,a_2]=1, Ba_1B=a_2 \rangle.$$
The groups $G(0,1), G(0,2)$ and $G(1)$ are pairwise non-isomorphic and turn out to be semidirect products, as in the case $p \neq 3$. Note that $G(0,1) \cong \mathbf{G}_3$ and therefore the Riemann surfaces  $C_3(\Lambda)/K(0,1)$ form the family $\mathscr{C}_3.$ 

Finally, it is worth noticing that there are two topologically inequivalent actions of $\mathbf{G}_3$ in genus 10, as can be obtained by using the routines given in \cite{BRR13}. One of them is the one represented by $K(0,1)$ and the other is not obtained here, as in this case the corresponding  $\mathbb{Z}_3^2$-quotient has genus two.  
\end{example}

Now, we can extend the previous two examples to the general case.

\begin{prop}
Let $p \geqslant 5$ be a prime number. If  $\mathcal{F}$ is a complex one-dimensional family of compact Riemann surfaces of genus $(p-1)(2p-1)$ endowed with a group of automorphisms $G$ isomorphic to a semidirect product of the form $\mathbb{Z}_p^2 \rtimes \mathbf{D}_3$, then $G \cong \mathbf{G}_p.$ In particular,  family of Kuribayashi-Komiya curves $\mathscr{C}_p$ corresponds to one irreducible component of $\mathcal{F}$.
\end{prop}

\begin{proof}
By the Riemann-Hurwitz fomula, if $S$ belongs to a complex one-dimensional family $\mathcal{F}$ as in the statement of the proposition, then $(S, N)$ is a $\mathbb{Z}_p^2$-action of signature $(0; p^6),$ where $N$ is the $p$-Sylow subgroup of $G.$ Then $S \cong C_p(\Lambda)/K$ and $G \cong Q_{S, G}/K$ for some $K \in \mathscr{C}_p(Q_{S,G}).$ We now study the quotients $Q_{S, G}/K$.

\noindent
{\bf 1.} Assume $K=K(l):=\langle a_1a_2a_3, a_1^la_2^{-1}, a_4^{l}a_6^{-1}\rangle,$ where $p \equiv 1 \mbox{ mod }3$ and $l \in \{1, \ldots, p-1\}$ satisfies $l^2+l+1 \equiv 0 \mbox{ mod }p$. In this case, $G(l):= Q_{S, G}/K(l)=\langle a_1, \ldots, a_6, A, B\rangle/\langle a_1a_2a_3=1, a_6=a_4^l, a_2=a_1^l \rangle.$
Note that in the quotient $a_2=a_1^l, a_3=a_1^{l^2}, a_4=a_6^{l^2}, a_5=a_6^l .$ It follows that 
$G(l) \cong \langle a_1, a_2, A, B : \ldots, Aa_1 A^{-1}=a_1^l, A a_2 A^{-1}= a_2^{l^2}, B a_1 B= a_2  \rangle.$
\noindent
{\bf 2.} Assume $K=K(r, s):=\langle a_1a_2a_3, a_1^r a_2^s a_4^{-1}, a_1^{-s}a_2^{r-s}a_5^{-1} \rangle$ where $r, s \in \{0, \ldots, p-1\}$ satisfy $r^2+s^2-rs \equiv 1 \mbox{ mod }p$. In this case, one has that $$G(0,1)=Q_{S,G}/K(0,1) \cong \langle a_1, a_2, A, B : \ldots, Aa_1 A^{-1}=a_2, A a_2 A^{-1}= (a_1 a_2)^{-1}, B a_1 A= a_2 \rangle.$$ $$G(0,p-1) =Q_{S,G}/K(0,p-1) \cong \langle a_1, a_2, A, B : \ldots, Aa_1 A^{-1}=a_2, A a_2 A^{-1}= (a_1 a_2)^{-1}, B a_1 B = a_2^{-1} \rangle.$$$$G(1,0)=Q_{S,G}/K(1,0)  \cong \langle a_1, a_2, A, B : \ldots, Aa_1 A^{-1}=a_2, A a_2 A^{-1}= (a_1a_2)^{-1}, B a_1 B = a_1, B a_2 B = (a_1a_2)^{-1}\rangle.$$$$G(p-1,0) =Q_{S,G}/K(p-1,0) \cong \langle a_1, a_2, A, B : \ldots, Aa_1 A^{-1}=a_2, A a_2 A^{-1}= (a_1a_2)^{-1}, B a_1 B= a_1^{-1}, B a_2 B = a_1 a_2 \rangle,$$ and for $r, s \neq 0$ we have that $G(r,s)=Q_{S,G}/K(r,s) \cong\langle a_1, a_2, A, B : \ldots, A a_1 A^{-1}=a_1^l, A a_2 A^{-1}= a_2^{l^2}, B a_1 B=  a_2 \rangle.$

Observe that the elements $\sigma_1:=a_2$ and $\sigma_2:=(a_1a_2)^{-1}$ of $K(0,1)$ generate a subgroup of it isomorphic to $K(1,0).$ Likewise, it can be seen that $G(l)\cong G(0,1)\cong G(0, p-1) \cong G(1,0)\cong G(p-1,0) \cong G(r,s).$ 
Finally, by considering the elements $\sigma_1:=a_1^2a_2$ and $\sigma_2:=a_1^{-1} a_2^{-2}$ of $K(0,1)$, we conclude that $G \cong \mathbf{G}_p$.
\end{proof}

\begin{example} 
The following table summarizes the number $N$ of topologically distinct triples $(S, N, \mathbf{G}_p )$ where $(S, N)$ is a $\mathbb{Z}_3^2$-action of signature $(0; 3^6)$ and the signature of $S/\mathbf{G}_p$ is $(0; 2,2,3,p)$, for some small primes (c.f. \cite[Proposition 2]{MRC}). 
{\small
$$
\begin{array}{|c|c|c|c|c|c|c|c|c|c|}\hline
p &  5 & 7 & 11 & 13 & 17 & 19 & 23 & 29 & 31 \\\hline
N &  2 & 3 & 3 & 4 &4 & 5 & 5 & 6 & 7 \\\hline
\end{array}
$$
}
\end{example}

\subsection{The case $L=\mathbb{Z}_2^2$} Note that if $p \neq 2$ then $G$ is isomorphic to a semidirect product of the form $\mathbb{Z}_p^2 \rtimes \mathbb{Z}_2^2.$ Set  $$L \cong G/N \cong \mathbb{Z}_2^2= \langle a, b: a^2=b^2=(ab)^2=1\rangle.$$

If $\infty, 0, 1, q_4, q_5, q_6$ are the cone points of $S/N$ then, after conjugating by a suitable M\"{o}bius transformation, we can assume $a(z)=-z, \, b(z)=\tfrac{\lambda}{z} \mbox{ and that } q_4=\lambda, \,q_5=-1, \,q_6=-\lambda.$ Note that $\Lambda=(q_4, q_5, q_6) \in \Omega_5$ if and only if $\lambda \neq 0, \pm 1.$ It follows that $S \cong C_{p}(\Lambda)/K_S$ and $N \cong H/K_S$ for some $K_S \in \mathcal{F}(p,5,2),$ where 
\begin{equation*}
C_p(\Lambda) :  \left \{ \begin{array}{ccccccc}
x_1^p & + & x_2^p & + & x_3^p & =  & 0\\
\lambda x_1^p & + & x_2^p & + & x_4^p & =  & 0\\
- x_1^p & + & x_2^p & + & x_5^p & =  & 0\\
-\lambda  x_1^p & + & x_2^p & + & x_6^p & =  & 0\\
\end{array} \right\} \subset \mathbb{P}^5\end{equation*}The group $L$ lifts to a group of automorphisms $Q_{S,G}$ of $C_p(\Lambda)$
  such that $Q_{S,G}/K_S \cong G$ and $S/G \cong C_{p}(\Lambda)/{Q_{S,G}}.$

\begin{lemm}\label{cable} The structure of the group $Q_{S,G}$ is as follows.
\begin{enumerate}
\item If $p \geqslant 3$ is a prime number, then $Q_{S,G}$ is isomorphic to a semidirect product $$\mathbb{Z}_p^5 \rtimes  \mathbb{Z}_2^2 =\langle a_1, \ldots, a_5 : a_j^p=[a_i, a_j]=1 \rangle \rtimes \langle A, B : A^2=B^2=(AB)^2=1\rangle,$$where the action by conjugation of $A$ and $B$ on $a_1, \ldots, a_5$ is given by$$A: (a_1, a_2, a_3, a_4, a_5) \mapsto (a_1, a_2, a_5, (a_1\cdots \, a_5)^{-1}, a_3) \, \mbox{ and } \, B:  (a_1, a_2, a_3, a_4, a_5)  \mapsto (a_2, a_1, a_4, a_3, (a_1\cdots \, a_5)^{-1}).$$

\item If $p=2$, then $Q_{S,G}$ is isomorphic to $$\langle a_{1},\ldots,a_{5},A,B: a_{j}^{2}=[a_i, a_j]=B^{2}=1, A^{2}=a_{1}, (AB)^{2}=a_{4}a_{5}, Aa_{1}A^{-1}=a_{1}, Aa_{2}A^{-1}=a_{2}, $$ $$ Aa_{3}A^{-1}=a_{5}, Aa_{4}A^{-1}=a_{1}a_{2}a_{3}a_{4}a_{5},
Ba_{1}B=a_{2}, Ba_{3}B=a_{4}, Ba_{5}B=a_{1}a_{2}a_{3}a_{4}a_{5}\rangle.$$
\end{enumerate}
\end{lemm}
\begin{proof}
Every lifting of $a$ is of the form
$$A([x_1: \ldots : x_6])=[x_1:\alpha_{2} x_2:\alpha_{3} x_5:\alpha_{4} x_6: \alpha_{5} x_3:\alpha_{6}x_4] \mbox{ where }
\alpha_{2}^{p}=\alpha_{3}^{p}=\alpha_{4}^{p}=\alpha_{5}^{p}=\alpha_{6}^{p}=-1,$$
and every lifting of $b$ is of the form
$$B([x_1: \ldots : x_6])=[x_2:\beta_{2} x_1:\beta_{3} x_4:\beta_{4} x_3:\beta_{5} x_6:\beta_{6} x_5] \mbox{ where }
\beta_{2}^{p}=\beta_{4}^{p}=\lambda, \beta_{3}^{p}=1, \beta_{5}^{p}=-1, \beta_{6}^{p}=-\lambda.$$

\noindent
{\bf 1.} If $p \geqslant 3$ then we may take $\alpha_{2}=\alpha_{3}=\alpha_{4}=\alpha_{5}=\alpha_{6}=-1$ and 
$\beta_{2}=\beta_{4}=\eta, \beta_{3}=1, \beta_{5}=-1, \beta_{6}=-\eta$ where $\eta$ is a fixed choice of $\lambda^{1/p}$, to get
$$A([x_1: \ldots : x_6])=[-x_1:x_2:x_5:x_6:x_3:x_4] \mbox{ and }B([x_1: \ldots : x_6])=[x_2:\eta x_1:x_4:\eta x_3:-x_6:-\eta x_5],$$
which satisfy the relations $A^{2}=B^{2}=(AB)^{2}=1$.

\noindent
{\bf 2.} If $p=2$ then we may take $\alpha_{2}=\alpha_{3}=\alpha_{4}=\alpha_{5}=\alpha_{6}=i$ and  
$\beta_{2}=\beta_{4}=-\eta, \beta_{3}=1, \beta_{5}=-i, \beta_{6}=-i\eta$ where $\eta$ is a fixed choice of $\lambda^{1/2}$, 
to get
$$A([x_1: \ldots : x_6])=[-ix_1:x_2:x_5:x_6:x_3:x_4] \mbox{ and }B([x_1: \ldots : x_6])=[x_2:-\eta x_1:x_4:-\eta x_3:-ix_6:-i\eta x_5],$$
which satisfy the relations
$A^{2}=a_{1}, B^{2}=1, (AB)^{2}=a_{4}a_{5}$.
\end{proof}

\begin{rema}\label{dd} Observe that if $\lambda=\lambda_{0} =i$ and $\Lambda_1=(i, -1, -i)$, then the M\"{o}bius transformation $d(z)= iz$ satisfies that $d^2=a$ and lifts to an automorphisms $D$ of order 4 of $C_p(\Lambda_1)$ such that $\langle D, B\rangle \cong \mathbf{D}_4$ and $\langle Q_{S,G}, D\rangle \cong \mathbb{Z}_p^5 \rtimes \mathbf{D}_4$. The action by conjugation of $D$ on $a_1, \ldots, a_5$ is given by $D: (a_1, a_2, a_3, a_4, a_5) \mapsto (a_1, a_2, a_4, a_5, (a_1 \cdots a_5)^{-1}).$
\end{rema}

Consider the group isomorphism $\sigma : \mbox{Aut}_g(H) \to \mathbf{S}_6=\mbox{Sym}\{a_1, \ldots, a_6\}$ and write $u=\sigma(A) =(3 \, 5)(4 \, 6) \mbox{ and } v:=\sigma(B)=(1 \, 2)(3 \, 4)(5 \, 6).$

\begin{prop}\label{nor2}
${\mathcal N}_{Q_{S,G}}$ is generated by $Q_{S,G}^*$ and the geometric automorphisms of $H$ given by 
$$\Psi_2(a_1, \ldots, a_6)=(a_1, a_2, a_3, a_6, a_5, a_4) \, \mbox{ and } \, \Psi_3(a_1, \ldots, a_6)=(a_1, a_2, a_4, a_3, a_6, a_5).$$
\end{prop}

\begin{proof}
We argue as done in Proposition \ref{nor}, after noticing that ${\mathcal N}_{Q_{S,G}}$ is isomorphic to the normalizer $N$ of $\langle u, v \rangle$ in $\mathbf{S}_6,$ which is given by 
$N=\langle u, v, \Psi_{2}=(4 \, 6), \Psi_{3}=(3 \, 4)(5 \, 6)\rangle$. Note that 
$N/\langle u, v \rangle \cong {\mathbb Z}_{2}^{2}$. 
\end{proof}

By Corollary \ref{coca7}, the number of pairwise topologically inequivalent actions $(S, N, G)$ in genus 
$(2p-1)(p-1)$ with a $\mathbb{Z}_p^2$-action of genus zero and $S/G$ of signature $(0; 2,2,p,2p)$ corresponds to the cadinality of $\mathscr{C}_p({Q_{S,G}})/{\mathcal N}_{Q_{S,G}}.$ The following proposition describes the members of $\mathscr{C}_p({Q_{S,G}}).$

\begin{prop} \label{verde4} If $p \geqslant 3$ is prime and  $K \in \mathscr{C}_p({Q_{S,G}})$ then one of the following statements holds.
\begin{enumerate}
\item There are $r, s \in \{0, \ldots, p-1\}$ satisfying that $r+s \in \{\tfrac{p-1}{2}, \tfrac{3p-1}{2}\}$ and $$K=K(r,s):=\langle a_1^ra_2^sa_3^{-1}, a_3a_5^{-1}, a_4a_6^{-1}\rangle.$$

\item $K$ agrees with one of the following
$$K_1:=\langle a_1a_2^{-1}, a_3 a_4^{-1}, a_5a_6^{-1} \rangle, \,\, K_2:=\langle a_1a_2^{-1}, a_3 a_6^{-1}, a_4a_5^{-1} \rangle,\,\, K_5:=\langle a_1a_2^{-1}, a_3 a_5^{-1}, a_4a_6^{-1} \, \rangle, \,\, K_6:=\langle a_1a_2, a_3 a_5^{-1}, a_4a_6^{-1} \rangle.$$

\item There is $r \in \{0, \ldots, p-1\}$ and either 
$$K=K_3(r):=\langle a_1^r a_3^{-1}a_4, a_1a_2, a_3a_6\rangle \,\, \mbox{ or } \,\, K=K_4(r):=\langle a_1^r a_3^{-1}a_6, a_1a_2, a_3a_4\rangle$$
\end{enumerate}

If $p =2$ and  $K \in \mathscr{C}_2({Q_{S,G}})$ then $K$ is one of the following
$$\bar{K}_1=\langle a_1a_2, a_3a_4, a_3a_5\rangle, \,\, \bar{K}_2=\langle a_1a_2, a_3a_4, a_1a_3a_5 \rangle,\,\,\bar{K}_3=\langle a_1a_2, a_3a_5, a_1a_3a_4\rangle \,\, \mbox{ or } \,\, \bar{K}_4=\langle a_1a_2, a_4a_5, a_1a_3a_4 \rangle,$$
\end{prop}
\begin{proof}
Let $K \in \mathscr{C}_p({Q_{S,G}})$ and let $\theta: H \to \mathbb{Z}_p^2$ be a  group epimorphism such that $K=\mbox{ker}(\theta).$ As $a_1 \notin K,$ we have that $\phi_1:=\theta(a_1)$ has order $p.$

\noindent
{\bf 1.} Assume $\theta(a_2) \notin \langle \phi_1 \rangle$ and write $\phi_2:=\theta(a_2)$ in such a way that $\langle \phi_1, \phi_2 \rangle \cong \mathbb{Z}_p^2.$ Set $\theta(a_3):=\phi_1^r\phi_2^s$ for $r,s \in \{0, \ldots, p-1\}$ that  are not simultaneously equal to zero. 

The fact that $A(a_3)=a_5, A(a_1)=a_1$ and $A(a_2)=a_2$ implies that $\phi_1^r\phi_2^s=\theta(a_3).$ Similarly, it can be seen that $\theta(a_6)=\theta(a_4).$ Now, the fact that $B(a_1)=a_2$ and $B(a_3)=a_4$ implies that $\theta(a_4)=\phi_2^r \phi_1^s.$ It follows that $\theta=(\phi_1, \phi_2, \phi_1^r\phi_2^s, \phi_1^s\phi_2^r, \phi_1^r\phi_2^s,  \phi_1^s\phi_2^r)$, and therefore $K=K(r,s).$ The relation $a_1 \cdots \, a_6=1$ implies that $2(r+s) \equiv -1 \mbox{ mod }p$ and therefore $r+s \in \{\tfrac{p-1}{2}, \tfrac{3p-1}{2}\}$.

\noindent
{\bf 2.} Assume $\theta(a_2) \in \langle \phi_1 \rangle$ and write $\theta(a_2)=\phi_1^{l}$ for some $l \in \{1, \ldots, p-1\}.$ We claim that $\theta(a_3) \notin \langle \phi_1 \rangle.$ In fact, otherwise, if we write $\theta(a_3)=\phi_1^m$ then $\theta(a_5)=\phi_1^{lm}$ and similarly $\theta(a_6)=\phi_1^{ml}.$ This contradicts the surjectivity of $\theta.$ Thus, we write  $\phi_2:=\theta(a_3)$ and therefore $\langle \phi_1, \phi_2 \rangle \cong \mathbb{Z}_p^2.$ Note that $\phi_1^l=\theta(a_2).$ As $B$ has order two, we deduce that $l=1$ or $l=p-1.$ Write $\theta(a_5)=\phi_1^r \phi_2^s$ for some $r,s \in \{0, \ldots, p-1\}$ that are not simultaneously zero. The equality $\theta(a_5)=\phi_1^r \phi_2^s$ coupled with the fact that $A$ has order two imply that $r(1+s)\equiv 0 \mbox{ mod }p,$ and that $s=1$ or $s=p-1.$

\noindent
{\bf 2.1} Assume $s=1.$ It follows that $r=0$ and $\theta(a_5)=\phi_2$. Write $\theta(a_4):=\phi_1^u\phi_2^v$ then $\theta(a_6)=\phi_1^u\phi_2^v$ and $u(l+v)\equiv 0 \mbox{ mod }p,$ and $v=1$ or $v=p-1.$ The fact that $a_1 \cdots \, a_6=1$ implies that $v=p-1$ and therefore 
\begin{enumerate}
\item $u=0$ and $l=p-1$, and then $\theta=(\phi_1, \phi_1^{-1}, \phi_2, \phi_2^{-1},\phi_2, \phi_2^{-1})$ and $K=K_6,$ or 
\item $u=p-1$ and $l=1$, and then $\theta=(\phi_1, \phi_1, \phi_2, \phi_1^{-1}\phi_2^{-1},\phi_2, \phi_1^{-1}\phi_2^{-1})$  and $K=K_5$.  
\end{enumerate}

\noindent
{\bf 2.2} Assume $s=p-1.$ It follows that $\theta(a_5)=\phi_1^r\phi_2^{-1}.$ If we write $\theta(a_4):=\phi_1^{u}\phi_2^{v}$ then, by proceeding as in the previous case, we obtain that $\theta(a_6)=\phi_1^{u+rv}\phi_2^{-v}$ where $v=1$ or $v=p-1, u(l+v) \equiv 0 \mbox{ mod } p$ and $2u \equiv r(l-v) \mbox{ mod }p.$ Moreover, The fact that $a_1 \cdots \, a_6=1$ implies that $1+l+r+2u+rv \equiv 0 \mbox{ mod }p.$

Assume $v=1.$ Then
\begin{enumerate}
\item $u=0, l=1, r=p-1$ and $\theta=(\phi_1, \phi_1, \phi_2, \phi_2, \phi_1^{-1}\phi_2^{-1}, \phi_1^{-1}\phi_2^{-1})$ and $K=K_1,$ or
\item  $l=p-1, u=-r$ and $\theta=(\phi_1, \phi_1^{-1}, \phi_2, \phi_1^{-r}\phi_2, \phi_1^{r}\phi_2^{-1},\phi_2^{-1})$ and $K=K_3(r).$
\end{enumerate}

Assume $v=p-1.$ Then
\begin{enumerate}
\item $u=r=p-1, l=1$ and $\theta=(\phi_1, \phi_1, \phi_2,  \phi_1^{-1}\phi_2^{-1}, \phi_1^{-1}\phi_2^{-1}, \phi_2)$ and $K=K_2,$ or

\item $l=p-1, u=0$ and $\theta=(\phi_1, \phi_1^{-1}, \phi_2, \phi_2^{-1}, \phi_1^{r}\phi_2^{-1},\phi_1^{-r}\phi_2)$ and $K=K_4(r).$

\end{enumerate}
The proof of the case $p=2$ is analogous. This case yields only four epimorphisms $\bar{\theta}: H \to \mathbb{Z}_2^2$ given by $$\bar{\theta}_1=(\phi_1, \phi_1, \phi_2, \phi_2, \phi_2, \phi_2), \,\, \bar{\theta}_2=(\phi_1, \phi_1, \phi_2, \phi_2, \phi_1\phi_2, \phi_1\phi_2),$$ $$ \bar{\theta}_3=(\phi_1, \phi_1, \phi_2, \phi_1\phi_2, \phi_2, \phi_1\phi_2) \,\, \mbox{ and } \,\, \bar{\theta}_4=(\phi_1, \phi_1, \phi_2, \phi_1\phi_2, \phi_1\phi_2, \phi_2),$$and the proof follows after noticing that $\bar{K}_j=\mbox{ker}(\bar{\theta}_j)$ for $j=1,2,3,4.$ 
\end{proof}

\begin{rema} It is worth emphasizing that whereas the epimorphisms $\bar{\theta}_2, \bar{\theta}_3$ and $\bar{\theta}_4$ define $\mathbb{Z}_2^2$-actions on Riemann surfaces $S$ of genus three that are topologically equivalent, they are not topologically equivalent as triples $(S, N, G).$
\end{rema}

\begin{theo}\label{pina} Let $p \geqslant 2$ be a prime number. The cardinality of $\mathscr{C}_2(Q_{S,G})/{\mathcal N}_{Q_{S,G}}$ is 3 if $p=2$, and $p+4$ otherwise. 
\end{theo}
\begin{proof}
By Proposition \ref{nor2}, we only need to count the number of orbits of the action of $\langle \Psi_2=(4 \, 6), \Psi_3=(3 \, 4)(5 \, 6) \rangle$  on $\mathscr{C}_p(Q_{S,G}).$ Assume $p \neq 2,$ and let $r, s \in \{0, \ldots, p-1\}$ such that  $r+s \in \{\tfrac{p-1}{2}, \tfrac{3p-1}{2}\}$. Note that $K(r,s)$ remains invariant under the action of $\Psi_2,$ and $\Psi_3(K(r,s))=\langle a_1^ra_2^sa_4^{-1}, a_3a_5^{-1}, a_4a_6^{-1}\rangle.$
If $\tau \in \{1, \ldots, p-1\}$ satisfies that $2 \tau \equiv 1 \mbox{ mod }p$ then one has that $(a_1^ra_2^sa_4^{-4} \cdot (a_3a_5^{-1} \cdot  a_4a_6^{-1})^{\tau})^{-1}=a_1^sa_2^ra_3^{-1}.$ This shows that $\Psi_3(K(r,s))=K(s,r)$ and therefore these groups form $\alpha$ orbits, where $\alpha$ is the number of pairs $(r,s)$ such that $r,s \in \{0, \ldots, p-1\}$ satisfies $r \leqslant s$ and $r+s \in \{\tfrac{p-1}{2}, \tfrac{3p-1}{2}\}$. A routine computation shows that $\alpha=\tfrac{p+1}{2}$. Similarly, if $r \in \{0, \ldots, p-1\}$ then it can be seen that $\Psi_2(K_3(r))=K_4(r)$ and $\Psi_3(K_3(r))=K_3(p-r)$. Thus, the groups $K_3(r)$ and $K_4(r)$ form $\tfrac{p+1}{2}$ orbits, represented by $K_3(r)$ where $r \in \{0, \ldots, \tfrac{p-1}{2}\}.$ Finally,  $K_1$ and $K_2$ form a single orbit, whereas $K_5$ and $K_6$ form one orbit each. All the above show that there are precisely $p+4$ orbits.  By proceeding analogously, it can be seen that if $p=2$ then $\bar{K}_2$ and $\bar{K}_4$ form a single orbit, and $\bar{K}_1$ and $\bar{K}_3$ form one orbit each.  This completes the proof. 
\end{proof}

\begin{example}
Theorem \ref{pina} says that there are exactly three pairwise topologically inequivalent actions $(S, N, G)$ where $S$ is a compact Riemann surface of genus three, $(S, N)$ is a $\mathbb{Z}_2^2$-action of signature $(0; 2^6)$ and $G/N\cong \mathbb{Z}_2^2.$ The signature of $S/G$ is $(0; 2,2,2,4).$ We denote by $\mathcal{F}_j$ the complex-one dimensional family determined by $\bar{K}_j$ and by $G_j$ (instead of $G$) the corresponding group. We have that $G_j \cong Q_{S, G}/\bar{K}_j$ and therefore Lemma \ref{cable} implies that $$G_1 \cong \langle a_3,A,B: a_{3}^{2}=A^4=B^{2}= (AB)^{2}= [A, a_3]=[B, a_3]=1\rangle \cong \langle a_3 \rangle \times \langle A, B \rangle \cong \mathbb{Z}_2 \times \mathbf{D}_4$$ $$G_2 \cong \langle a_1, a_3,A,B: A^2=(AB)^2=a_1, a_1^2=a_3^2=[a_1,a_3]=B^2= (Aa_3)^2=[B, a_3]=[A, B]=1\rangle
 \cong \langle B \rangle \times \langle A, a_3 \rangle \cong \mathbb{Z}_2 \times \mathbf{D}_4$$ $$G_3\cong \langle a_3,A,B: a_{3}^{2}=A^4=B^{2}= (AB)^{2}=[a_3, A]=1, Ba_3B=Aa_3\rangle =\langle A, a_3 \rangle \rtimes \langle B \rangle \cong (\mathbb{Z}_4 \times \mathbb{Z}_2) \rtimes \mathbb{Z}_2.$$ 
 
As an application of Proposition \ref{pp5} we obtain the following algebraic descriptions for the families $\mathcal{F}_j$. 
\begin{equation*}
\mathcal{F}_1:\left\{ \begin{array}{l}
y_{1}^{2}=(x^2-1)(x^2-\lambda^2)\\
y_{2}^{2}=x\\
\end{array}
\right.  \,\,\, \mathcal{F}_2:\left\{ \begin{array}{l}
y_{1}^{2}=(x^2-1)(x^2-\lambda^2)\\
y_{2}^{2}=x(x+1)(x+\lambda)\\
\end{array}
\right. \,\,\, \mathcal{F}_3:\left\{ \begin{array}{l}
y_{1}^{2}=(x^2-1)(x^2-\lambda^2)\\
y_{2}^{2}=x(x^2-\lambda^2)\\
\end{array}
\right.
\end{equation*}where $\lambda \in \mathbb{C}-\{0, \pm 1\}.$ Observe that the  (generic) members of $\mathcal{F}_2$ and $\mathcal{F}_3$ are non-hyperelliptic, whereas the ones of $\mathcal{F}_1$ are; the hyperelliptic involution being $(x, y_2, y_2) \mapsto (x, -y_2, y_2)$. 

Note that $\bar{K}_1$ is $D-$invariant (see Remark \ref{dd}) and therefore the member of $\mathcal{F}_1$ obtained by taking $\lambda =i$ is endowed with a group of automorphisms of order 32 acting with signature $(0; 2,4,8)$, and is represented by $y_1^2=y_2^8-1$ (this Riemann surfaces is known as the Accola-Maclachlan curve of genus three).

We remark that there is only one topological class of actions of $\mathbb{Z}_2 \times \mathbf{D}_4$ in genus three with signature $(0; 2,2,2,4)$. In particular, if $S_j$ belongs to $\mathcal{F}_j$ then the pairs $(S_1, G_1)$ and $(S_2, G_2)$ are topologically equivalent. However, the triples $(S_1, N_1, G_1)$ and $(S_2, N_2, G_2)$ are not, as this example shows. 
\end{example}

We extend the previous example to the general case.

\begin{prop}
Let $p \geqslant 3$ be a prime number and let $S$ be a compact Riemann surface of genus $(2p-1)(p-1)$ endowed with a group of automorphisms $G$ isomorphic to a semidirect product $\mathbb{Z}_p^2 \rtimes \mathbb{Z}_2^2$. Let $N \cong \mathbb{Z}_p^2$ be the normal Sylow $p$-subgroup of $G.$  If $S/N$ has signature $(0; p^6)$ then one of the following statements hold.  
\begin{enumerate}
\item $G$ is isomorphic to $\mathbb{Z}_{2p} \times \mathbf{D}_p$ and there are precisely $\tfrac{p+5}{2}$ pairwise topological inequivalent triples $(S, N, G).$
\item $G$ is isomorphic to $\mathbf{D}_p \times \mathbf{D}_p$ and there are precisely $\tfrac{p+1}{2}$ pairwise topological inequivalent triples $(S, N, G).$
\item $G$ is isomorphic to $\mathbb{Z}_2 \times (\mathbb{Z}_p^2 \rtimes_{(-1,-1)} \mathbb{Z}_2)$ and there is only one  topological class of triples $(S, N, G).$
\end{enumerate}
\end{prop}

\begin{proof}
We only need to study the quotients $Q_{S, G}/K$, where, by Lemma \ref{cable}, we have that$$Q_{S, G} =\langle a_1, \ldots, a_6, A, B, : a_j^p=[a_i, a_j]= a_1 \cdots \, a_6= A^2=B^2=(AB)^2= 1, $$ $$[A, a_1]=[A, a_2]=1, Aa_3A=a_5, Aa_4A=a_6, Ba_1B=a_2,Ba_3B=a_4, Ba_5B=a_6\rangle,$$and $K$ runs over the subgroups given in Proposition \ref{verde4} up to ${\mathcal N}_{Q_{S,G}}$-action (see the proof of Proposition \ref{pina}). 

Let $r, s \in \{0, \ldots, p-1\}$ be integers satisfying that $r+s \in \{\tfrac{p-1}{2}, \tfrac{3p-1}{2}\}.$ Then $$G(r,s):=Q_{S, G}/K(r,s) \cong \langle a_1, a_2, A, B : \dots,  [A, a_1]=[A, a_2]=1, Ba_1B=a_2\rangle,$$ which is isomorphic to $\langle A, a_1a_2 \rangle \times \langle B, a_1a_2^{-1}\rangle \cong \mathbb{Z}_{2p}\times \mathbf{D}_p.$ Similarly, for each $r \in \{0, \ldots, \tfrac{p-1}{2}\}$ we have that $$G_3(r):=Q_{S, G}/K_3(r)\cong\langle a_1, a_3, A, B : \ldots, [A, a_1]=1, Aa_3A=a_1^ra_3^{-1}, Ba_1B=a_1^{-1}, Ba_3B=a_1^{-r}a_3 \rangle$$which is isomorphic to $\langle A, a_1^ra_3^{-2} \rangle \times \langle B, a_1\rangle \cong \mathbf{D}_{p}\times \mathbf{D}_p,$ and
$$G_6:=Q_{Q, S}/K_6 \cong \langle a_1, a_3, A, B: \ldots, [A, a_1]=[A, a_3]=1, Ba_1B=a_1^{-1}, Ba_3B=a_3^{-1}\rangle$$ is  isomorphic to $\langle A \rangle \times \langle a_1, a_3, B\rangle \cong \mathbb{Z}_2 \times (\mathbb{Z}_p \rtimes_{(-1,-1)}\mathbb{Z}_2).$ Finally,  by proceeding analogously, we obtain that $G_1=Q_{Q, S}/K_1$ and $G_5=Q_{Q, S}/K_5$ are isomorphic to $\mathbb{Z}_{2p} \times \mathbf{D}_p$ and the proof is done.
\end{proof}


%

\end{document}